\documentclass{article}
\usepackage[english]{babel}
\usepackage[margin=1.in]{geometry}

\usepackage[T1]{fontenc}
\usepackage[utf8]{inputenc}

\usepackage{amsmath,amsbsy,amssymb,amsthm,amsfonts}

\usepackage{graphicx} 
\usepackage[font=small,labelfont=bf]{caption}
\usepackage{lmodern}
\usepackage{microtype}

\usepackage{mathrsfs}
\usepackage{mathtools}
\usepackage{enumerate}
\usepackage{float}
\usepackage[dvipsnames]{xcolor}
\usepackage{dsfont}

\usepackage{multirow}
\usepackage{empheq}
\usepackage{listings}
\usepackage{color}
\usepackage{subcaption}
\usepackage[symbol]{footmisc}

\usepackage[square,comma,numbers]{natbib}

\usepackage[colorlinks=true, allcolors=black]{hyperref}
\usepackage[linesnumbered, vlined]{algorithm2e}
\usepackage{etoolbox}

\makeatletter
\newcommand{\Max}[1]{\@latex@warning{Remaining comment: #1}\textcolor{orange}{#1}}
\makeatother

\numberwithin{equation}{section}
\usepackage{tikz}

\def\d{\mathrm{d}}
\def\eps{\epsilon}
\def\pmu{\delta_\nu}
\def\pme{\delta_\eta}
\def\pmmu{\delta_\mu}
\def\om{\omega}
\def\Om{\Omega}
\def\Omc{\Omega^c}

\def\be{\begin{equation}}
\def\ee{\end{equation}}

\def\Pd{\P^\delta}
\def\Ed{\E^\delta}

\def\fm{\mathfrak{m}}
\def\fs{\mathfrak{s}}

\def\bea{\begin{align}}
\def\eea{\end{align}}

\def\bea*{\begin{align*}}
\def\eea*{\end{align*}}

\def\cref#1{Figure~\ref{#1}}

\def\sA{{\mathcal A}}

\def\sC{{\mathcal C}}

\def\cG{{\mathcal G}}

\def\cI{{\mathcal I}}
\def\cJ{{\mathcal J}}

\def\cL{{\mathcal L}}

\def\cS{{\mathcal S}}

\def\cT{{\mathcal T}}

\def\E{\mathbb{E}}
\def\F{\mathbb{F}}

\def\P{\mathbb{P}}
\def\R{\mathbb{R}}

\def\T{\mathbb{T}}

\def\sA{{\mathscr{A}}}
\def\sB{{\mathscr{B}}}
\def\sC{{\mathscr{C}}}
\def\sD{{\mathscr{D}}}
\def\sE{{\mathscr{E}}}
\def\sF{{\mathscr{F}}}
\def\sG{{\mathscr{G}}}

\def\sL{{\mathscr{L}}}
\def\sM{{\mathscr{M}}}
\def\sN{{\mathscr{N}}}
\def\sP{{\mathscr{P}}}
\def\sW{{\mathscr{W}}}
\def\sX{{\mathscr{X}}}
\def\sY{{\mathscr{Y}}}

\def\bN{\sN}
\def\bW{\sW}

\newtheorem{Def}{Definition}[section]
\newtheorem{Thm}[Def]{Theorem}
\newtheorem{Cor}[Def]{Corollary}
\newtheorem{Rmk}[Def]{Remark}
\newtheorem{Lem}[Def]{Lemma}
\newtheorem{Prop}[Def]{Proposition}
\newtheorem{Asm}[Def]{Assumption}

\newtheorem*{Thm*}{Theorem}

\newcommand{\Prob}{\mathbb{P}}
\renewcommand{\d}{\mathrm{d}}
\newcommand{\wt}{\widetilde}


\numberwithin{equation}{section}
\usepackage{fancyhdr}
\title{Optimal Control and Potential Games in the Mean Field\footnote{The authors would like to thank Professors Dylan Possama\"i and Ludovic Tangpi for insightful
comments.  Research
partially supported by the National Science Foundation grant
 DMS 2406762. }}

\author{Felix H\"{o}fer\footnote{
Department of Operations Research and Financial
Engineering, Princeton University, Princeton, NJ, 08540, USA, email: 
{\tt fhoefer@princeton.edu}. }
\and H. Mete Soner\footnote{
Department of Operations Research and Financial
Engineering, Princeton University, Princeton, NJ, 08540, USA, email: 
{\tt soner@princeton.edu}. }}

\date{\today}

\begin{document}

\maketitle

\begin{abstract}
\noindent
We study a mean field optimal control problem with general non-Markovian dynamics, including both common noise and jumps. We show that its minimizers are Nash equilibria of an associated mean field game of controls. 
These types of games are necessarily potential, and the Nash equilibria derived as the minimizers of the control problem are closely connected to McKean-Vlasov equations of Langevin type.
To illustrate the general theory, we present several examples, including a mean field game of controls with interactions through a price variable, and mean field Cucker-Smale Flocking and Kuramoto models. We also establish the invariance property of the value function, a key ingredient used in our proofs.
\end{abstract}

\section{Introduction}

Mean field games (MFGs) were independently introduced by Lasry \& Lions \cite{lasry_jeux_2006,lasry_jeux_2006-1, lasry_mean_2007} and Huang, Caines, \& Malhamé \cite{huang_individual_2003,huang_invariance_2007, huang_large-population_2007, huang_nash_2007} as a framework to approximate large population games in which the interaction appears through the empirical distribution of agents. Since their introduction, they have been widely applied in contexts ranging from heterogeneous agent models in economics to models of jet-lag recovery in neuroscience.

In this work, we focus on a particular class of MFGs that emerge as the first-order optimality condition of a \emph{mean field control problem} (MFC). This connection was first identified in the seminal paper \cite{lasry_mean_2007} by Lasry \& Lions and has since been extensively studied under the name \emph{potential mean field games} as we review below. 
Our primary goal is to establish a general connection between potential MFG and MFC problems. In addition, we identify a close link to a class of Wasserstein gradient flows. The latter connection can be leveraged to construct compelling mean field games that have similar qualitative behavior as the gradient flows. In Section \ref{sec:examples}, we discuss two such examples: the MFG flocking model of \cite{nourian2010synthesis,nourian2011mean} 
and the Kuramoto MFG of \cite{carmona2023synchronization,YMMS,yin_synchronization_2010}.

While most existing studies on potential games, e.g.\ \cite{briani2018stable}, start from a MFG and establish structural conditions on the cost functions
under which equilibria can be obtained from a MFC problem,  we start from a general MFC problem and construct an associated potential MFG. 
By adopting a probabilistic formulation for both the control and game problems, we obtain general results that accommodate \emph{non-Markovian} dynamics under \emph{jumps} and \emph{common noise}. In the case of \emph{separable} cost structures, our framework recovers existing results, while in the \emph{non-separable} case, it leads to \emph{mean field games of controls}.

\subsection{MFC, MFG and Main Results}

\noindent
On a filtered probability space $(\Omega,\sF,\F,\P)$ 
satisfying the usual conditions,
let $W^0$ and $W$ be $\F$-Brownian motions,
and $N$ and $N^0$ be stationary $\F$-Poisson random measures on 
$[0,\infty)\times (\R^k\setminus\{0\})$. In our setting,
$W$ and $N$ represent the idiosyncratic noise,
while $W^0$ and $N^0$  model the common noise. We write $\wt N$ and $\wt N^0$ for the corresponding compensated random measures.
Let $T>0$ be the time horizon. For a given initial condition $X_0$, 
and coefficients $(b, \sigma, \sigma^0, \lambda, \lambda^0)$
satisfying Assumption \ref{asm:state} below,
the {\emph{controlled state process}} $X^\alpha \in \R^n$ 
is the unique strong solution to the following
stochastic differential equation,
\begin{equation}
\label{eq.state}
\begin{split}
X^{\alpha}_t = &X_0 + \int_0^t b(s,X^\alpha_{\cdot\land s},\alpha_s)\,\d s + \int_0^t \left(\sigma(s,X^\alpha_{\cdot\land s},\alpha_s)\,\d W_s 
+ \sigma^0(s,X^\alpha_{\cdot\land s},\alpha_s)\, \d W^0_s \right)\\
&   +  \int_0^t\int_{\R^k\setminus\{0\}} 
\left(  \lambda(s,X^\alpha_{\cdot\land s-},\alpha_s,\zeta)\, \wt N(\d s,\d \zeta)  + \lambda^0(s,X^\alpha_{\cdot\land s-},\alpha_s,\zeta)\, \
\wt N^0(\d s,\d \zeta)\right),
\end{split}
\end{equation}
where $\alpha_\cdot$ is  any $\F$-predictable,
$\d\P\otimes \d t$-square integrable \emph{control} process taking values 
in a Borel subset $U$ of some Euclidean space. 
Let $\sA$ denote the set of all such control processes.
Notice that the above dynamics do not contain any mean field interactions.
We refer to Section \ref{ssec:law-dep} for a class of control problems in which some of the dynamics are allowed to be of mean field type.

Consider the cost functionals,
$$
F : [0,T] \times \sD_n \times U \times \sP(\sD_n\times U) \to \R,
\quad
G : \sD_n \times \sP(\sD_n) \to \R,
$$
where $\sD_n$ denotes the Skorokhod space of c\`adl\`ag functions from
$[0,T]$ to $\R^n$, and for any measurable space $S$,
$\sP(S)$ is the set of
probability measures on $S$. We set\footnote{We note that the terminal cost $G$ depends on $\cL^\alpha_T$ only through its first marginal.}
\begin{equation}
\label{eq:MFC_cost1}
J(X_0,\alpha) := \E \left[\int_0^T 
F(t, X^{\alpha}_{\cdot\wedge t},\alpha_t,\cL^\alpha_t)\,\d t 
+ G(X^{\alpha}_{\cdot\land T}, \cL^\alpha_T)\ \right],
\end{equation}
where $\cL^\alpha_t:= \cL((X^{\alpha}_{\cdot\land t}, \alpha_t) \,|\, W^0,N^0)$
is the regular conditional distribution of $(X^{\alpha}_{\cdot\land t}, \alpha_t)$
given the common noise $(W^0,N^0)$.
Then, the \emph{mean field control} (MFC) problem consists of minimizing $J(X_0,\alpha)$
over all $\alpha\in\sA$.

We next introduce the associated {\emph{mean field game of controls}.
Let $\Pi: \sD_n\times U\to \sD_n $ be the projection on the first coordinate, and for  
$(t,x,a,\nu)\in[0,T]\times \sD_n\times U\times
\sP( \sD_n\times U)$, let $\mu =\Pi_\sharp\nu$ and set
\begin{equation} 
    \label{eq:pot-game-cost}
    \begin{split}
 	f(t,x,a,\nu) &:= \int_{\sD_n\times U}  \pmu 
        F (t,\tilde x,\tilde{a}, \nu)(x,a)\,\nu(\d \tilde x,\d \tilde{a}) 
        + F(t,x,a,\nu),\\
        g(x, \nu) &:= \int_{\sD_n} 
        \pmmu G(\tilde x, \mu)(x)\, \mu(\d \tilde x)  
         + G(x,\mu).
    \end{split}
\end{equation} 
Given an adapted stochastic process $\nu^*=(\nu^*_t)_{t \in [0,T]}$ with values in 
$\sP(\sD_n \times U)$, we define the \emph{game cost} of a typical agent by,
\begin{equation}
\label{eq:MFG-cost}
J_g(X_0,\alpha, \nu^*) := 
\E\left[\int_0^T f(t, X^{\alpha}_{\cdot\wedge t}, \alpha_t,\nu^*_t)\,\d t 
+ g(X^{\alpha}_{\cdot\land T}, \nu^*_T )\right].
\end{equation}
Then, a \emph{Nash equilibrium} of the associated MFG  of controls
is a pair of an admissible control process $\alpha^*\in \sA$
and a process $\nu^*=(\nu^*_t)_{t \in [0,T]}$ with values in 
$\sP(\sD_n \times U)$ which satisfy,
\be
\label{eq:pot_game}
J_g(X_0,\alpha^*, \nu^*)  =   \inf_{\alpha\in\sA}\,\, J_g(X_0,\alpha, \nu^*)
\ \ \text{and}\ \
\nu^*_t 
=\cL((X^{\alpha^*}_{\cdot\land t}, \alpha^*_t) \,|\, W^0,N^0), \ \  \forall \ t\in[0,T].
\ee

The main result of the paper connects the MFC and MFG problems.
\begin{Thm*}
    Under natural regularity and growth
    conditions on the coefficients,  
    for any minimizer $\alpha^*$ of the MFC problem \eqref{eq:MFC_cost1},  the pair
     $(\alpha^*,\nu^*)$ with 
     $\nu^*_t:=\cL((X^{\alpha^*}_{\cdot\land t}, \alpha^*_t) \,|\, W^0,N^0)$ 
     for $t\in[0,T]$,
    is a Nash equilibrium of the MFG of controls \eqref{eq:pot_game}.
\end{Thm*}

The precise formulation is provided in Theorem \ref{thm:pot_game_finite} below. 
Additionally, in Sections \ref{sec:main} and \ref{sec:comparison_Cardaliaguet} below, we discuss specific structures which show that this theorem covers previously known results in the literature where dynamics are independent of the law. An extension to law-dependent dynamics is treated in Subsection  \ref{ssec:law-dep}. Furthermore, it is clear that the results of this paper remain true in the absence of the common noise by simply setting $\sigma^0=\lambda^0=0$ and in the absence of jumps by setting $\lambda=\lambda^0=0$. As discussed in Subsection \ref{ssec:markov},  in the classical separable models, we obtain a standard MFG depending only on the distribution of the state process and not the control.

A key technical tool in our analysis is the invariance principle presented in
Theorem \ref{thm:invariance}. This result states that both the value function of the MFC problem and the Nash equilibria of the MFG are independent of the underlying probabilistic structure. 
This property is essential to our approach, as our proof requires extending the original probabilistic basis by adding independent Bernoulli random variables.
The invariance principle has been established for diffusions in \cite{Cosso2023,McKean_Vlasov_DPP}, and
we extend it to our setting.

\subsection{Context and Earlier Results}

Following the introduction of potential mean field games (MFGs) in \cite[Section 2.6]{lasry_mean_2007}, 
they have been developed in \cite{briani2018stable, cardaliaguet2015weak, cardaliaguet2015mean, cardaliaguet2015second, orrieri2019variational, graber2021weak, carmona2025conditional} in both the first-order (i.e., without Brownian noise) and second-order  settings. 
The majority of the existing literature on potential mean field games starts from a characterization 
of the Nash equilibria as solutions of a forward-backward system of partial differential equations. 
These works interpret the PDE system as an optimality condition by using an elegant application of the Fenchel-Rockafellar Duality 
Theorem or the von Neumann Minimax Principle, as illustrated in \cite{orrieri2019variational}.

For a comprehensive introduction to these techniques — especially in the first-order case — we refer the reader to the recent work \cite{graber2024remarks}, which also provides historical context.  
Related analyses appear in the context of optimal transport in \cite{cardaliaguet2013geodesics}
and for (first-order) \emph{variational MFGs} 
in \cite[Section 2]{achdou2020lecture}.
Extensions of these results include infinite-horizon cost criteria \cite{masoero2019long} and problems with state constraints \cite{daudin2022optimal}. 
The structure of potential MFGs has been leveraged to study weak solutions of the master equation \cite{cecchin2022weak} and to prove the convergence of learning procedures to solutions of the MFG system \cite{briani2018stable, cardaliaguet2017learning}. Applications of potential MFGs include games for market competition \cite{graber2018variational, bonnans2021schauder}.
We also refer to \cite{guo2023} for the study of stochastic differential potential games with finitely many players.

On finite state spaces, direct arguments such as the 
Pontryagin principle imply that potential MFG equilibria can be viewed 
as first-order conditions of an associated control problem
as initiated in  \cite{gomes2013continuous, gueant2011infinity}. 
For finite state problems with no uniqueness, \cite{cecchin2022selection} proposes a 
selection principle of MFG equilibria as the minimizers of the associated control problem when uniqueness fails. 
In the continuous state case, an approach to potential MFGs that 
is not based on convex duality is presented in the classical book 
\cite{carmona_probabilistic_2018-1}.  
Chapters 6.2.5 and 6.7.2 of \cite{carmona_probabilistic_2018-1}
identify the Pontryagin system of the MFC problem with the
FBSDE system describing MFG Nash equilibria in the probabilistic setting. The forthcoming paper
\cite{carmona2025conditional} discusses potential mean field games in the weak formulation.

We work in a general stochastic formulation of the control problem
using direct methods, rather than the FBSDEs.  
Our results are in considerable generality,
extending to problems featuring $(i)$ common noise, $(ii)$ jumps, $(iii)$ non-separable cost structures and $(iv)$ coefficients that depend on the entire path of the state variable. The non-separability of the cost functions leads to a \emph{mean field game of controls}, 
where the interaction between agents is through both the state and control
variables. Potential MFGs of controls have been studied in \cite{graber2021weak, bonnans2021schauder} in their analytic formulation. We provide further discussion  of the deterministic and probabilistic formulations in Subsection
\ref{sec:comparison_Cardaliaguet}.\\

The paper is organized as follows. In Section 3, after reviewing  notation, we define the mean field control (MFC) problem and the mean field game (MFG) of controls, and we present our main result. We then discuss the result and several special cases in Section \ref{sec:discuss}, and 
the proof of our main result Theorem \ref{thm:pot_game_finite}
is given in Section \ref{sec:proof}.  
Examples of potential mean field games of controls and the connection between the Langevin dynamics,
MFC, and MFG are presented in Section \ref{sec:examples}. 
Section \ref{sec:invariance} 
proves the invariance principle.
Appendix \ref{app:point} provides a
summary of standard results and notations on 
point processes, and Appendix \ref{sec:dih}
discusses the discounted infinite horizon setting.

\section{Notation and Setting}

For a topological space $S$, $\sB(S)$ is the set
of all Borel sets.
The dimension of the state process is denoted by $n$,
$d$ refers to the dimension of the idiosyncratic 
Brownian motion,  and $\ell$ is the dimension 
of the common Brownian noise. 
$T>0$ is the horizon. $\sD_k(0,T)$ is the Polish space of c\`{a}dl\`{a}g functions 
from $[0,T]$ to $\R^k$ endowed it with the Skorokhod $J_1$ topology.
When the domain is clear, we write $\sD_k$.
On $\sD_k$, we sometimes utilize the supremum norm 
$\|x\|_T:=\sup_{t\in[0,T]}|x(t)|$. 
The \emph{control} set $U$ is a non-empty Borel
subset of a Euclidean space. 
On any Polish space $S$, $\sP(S)$
is the set of probability measures, and
$\sP_2(S)$ is  the Wasserstein space of measures 
with finite second moments. 
We will frequently consider measures on the product space $\sD_k\times U$ and will generically denote by $\Pi(x,a)=x$, $(x,a)\in\sD_k\times U$, the projection on the first coordinate. 
If $\nu\in\sP(\sD_k\times U)$, then $\Pi_\sharp\nu$ refers to the push-forward of $\nu$ under $\Pi$.
A function $\phi:\sP_2(S)\to\R$ is called (linearly) differentiable,
if there exists a measurable function $\pme \phi :\sP_2(S)\times S\to\R$ 
which is of at most quadratic growth in $x\in S$ and satisfies, 
\begin{equation*}
    \phi(\mu)-\phi(\eta)= \int_0^1 
    \int_S \pme \phi(\eta+\tau (\mu-\eta))(s) \, (\mu-\eta)(\d s)\, \d \tau,
    \qquad \forall \mu, \eta \in\sP_2(S).
\end{equation*}
If $\phi:\sP_2(S)\to\R^k$ is vector-valued, then we say that $\phi$ is differentiable if all its components are differentiable. In this case we set, $\delta_\eta \phi = (\delta_\eta \phi_1,\ldots,\delta_\eta\phi_k)$. For a $\sigma$-algebra $\sG$ and an $S$-valued
random variable $Y$, the random measure $\cL(Y \,| \,  \sG) \in \sP(S)$
is the regular condition distribution of $Y$ given $\sG$. 
We write $\cL^\P=\cL$, when the dependence on the underlying probability measure $\P$ is important.
When $\sG$ is trivial, $\cL(Y)$ is the \emph{law} of $Y$. Finally, $\sL^2(\sG)$ denotes the set of all square-integrable $n$-dimensional random variables which are $\sG$-measurable.
\subsection{Probabilistic Structures}
We follow the standard terminology of \cite{Ikeda_Watanabe} and
provide further details in Appendix \ref{app:point}.

On the \emph{jump space} $(E,\sE) := (\R^k\setminus\{0\},\sB(\R^k\setminus\{0\}))$ we fix characteristic jump measures $n(\d\zeta)$ and $n^0(\d\zeta)$ of the idiosyncratic jumps $N$ and common jumps $N^0$, respectively. 

\begin{Def}\label{def:prob_basis}
{\rm{A}} reference probability basis {\rm{$\gamma$ is a tuple
$\gamma = (\Omega, \sF, \P, \F, \sW, \sN)$ given by,
\begin{enumerate}[(i)]
\item $(\Omega, \sF, \P, 
\F = (\sF_t)_{t\geq0})$ is a filtered 
probability space satisfying the usual conditions;
 \item $\sW=(W^0,W)$, where $W^0$ and
$W $  are  $\F$-Brownian motions in $\R^\ell$ and $\R^d$, respectively;
\item $\sN=(N^0,N)$, where
$N^{0}$ and $N$ are Poisson random measures associated to stationary 
$\F$-Poisson stationary point processes with characteristic jump measures $n^0(\d\zeta)$ and $n(\d\zeta)$, respectively. 
\end{enumerate}
We assume that the corresponding point processes 
$p^0$ and $p$
have $\P$-almost surely disjoint domains, or equivalently
their jump times are almost surely disjoint. We further assume that $(\sW,\sN)$ are independent. Finally, let $\sG$ denote the completion of the sigma-algebra
generated by the common noise processes $(W^0,N^0)$. }}
\end{Def}

\begin{Def}[Admissible controls]\label{def:adm_controls}
{\rm{The set of}} admissible controls  
{\rm{$\sA$ is the collection of all $\F$-predictable
processes $\alpha:\Omega\times [0,T] \to U$
that are square integrable, i.e.,
\be
 \label{eq:norm}
\|\alpha\|_{2,T}^2:= \E\int_0^T |\alpha_t|^2 \ \d t < \infty.
\ee}}
\end{Def}

When the dependence on the basis $\gamma$ is important, we write $\sA(\gamma)$,  $\cG(\gamma)$,
$$
\gamma=(\Om^\gamma,\sF^\gamma,\P^\gamma,
\F^\gamma=(\sF_t^\gamma)_{t \in [0,T]},\sW^\gamma,\sN^\gamma),
\ \ \E^\gamma= \E^{\P^\gamma}, \ \text{and} \ \ \cL^\gamma= \cL^{\P^\gamma},
$$
and we suppress this dependence in our notation when it is clear or redundant.

\subsection{Controlled State Process}

In this subsection, we state standard conditions 
on the coefficients of
the equation \eqref{eq.state}:
\begin{equation*}
    (b,\sigma,\sigma_0):[0,T]\times \sD_n\times U
    \to \R^n \times \sM_{n\times d}\times
     \sM_{n\times\ell},
\end{equation*}
\begin{equation*}
    (\lambda,\lambda^0): [0,T] \times \sD_n \times U \times E
    \to \R^n\times \R^n,
\end{equation*}
where $ \sM_{m\times m'}$ is the set of
$m\times m'$ matrices.

\begin{Asm}\label{asm:state}
All coefficients are measurable and  
there exist constants $K,L>0$ and $K(\zeta),L(\zeta)$ 
such that for any $(t,x,x',a,\zeta)\in [0,T] \times 
\sD_n^2 \times U\times E$ we have,
\begin{align*}
    |(b,\sigma,\sigma^0)(t,x,a)-(b,\sigma,\sigma^0)(t,x',a)| &\leq L\|x-x'\|_t,\\
    |(\lambda,\lambda^0)(t,x,a,\zeta)-(\lambda,\lambda^0)(t,x',a,\zeta)| &\leq L(\zeta)\|x-x'\|_t,
\end{align*}
$$
|(b,\sigma,\sigma^0)(t,x,a)|^2 \leq K(1+\|x\|_t^2+|a|^2),\ \
    |(\lambda,\lambda^0)(t,x,a,\zeta)|^2 \leq K(\zeta)^2(1+\|x\|_t^2+|a|^2),
$$ 
and the constants $L(\zeta),K(\zeta)$ satisfy $\int_E (L(\zeta)^2+K(\zeta)^2)\, (n^0+n)(\d\zeta)<\infty$.
\end{Asm}

Set $\boldsymbol{\sigma} := (\sigma^0,\sigma)$, and 
$\boldsymbol{\lambda} := (\lambda^0,\lambda)$,
and we rewrite \eqref{eq.state} as,
\begin{equation}
\begin{split}
 X^\alpha_t = X_0 + \int_0^t &
 b(s,X^\alpha_{\cdot\land s},\alpha_s)\,\d s + \int_0^t 
 \boldsymbol{\sigma}(s,X^\alpha_{\cdot\land s},\alpha_s)\,\d \bW_s\\ 
&+ \int_0^t\int_E 
\boldsymbol{\lambda}(s,X^\alpha_{\cdot\land s-},\alpha_s,\zeta)\, 
\wt\bN (\d s,\d \zeta),\quad t\geq0.
\end{split}
\label{eq:state} 
\tag{SDE$(X_0,\alpha)$}
\end{equation}

Under these assumptions, the standard Picard iteration technique implies the following theorem, see \cite[Theorem 3.1]{kunita_jumps}.

\begin{Thm}\label{thm:state}
Under Assumption \ref{asm:state}, for any
probability basis $\gamma$, 
admissible control $\alpha\in\sA(\gamma)$,
and $X_0\in \sL^2(\sF^\gamma_0)$, 
there exists a unique $\mathbb{F}^\gamma$-adapted 
c\`{a}dl\`{a}g process $X^\alpha$ that satisfies \ref{eq:state}. Furthermore, 
$\E^\gamma[\sup_{0\leq s\leq t}|X_s^{\alpha}|^2]<\infty$ for all $t \ge 0$.
\end{Thm}

\section{Main Result}
\label{sec:main}

In this section, we first define the MFC and 
MFG problems and then state the main result.

\subsection{The Mean Field Control Problem}
\label{sec:mfc}

Let a probability basis $\gamma$ be given and 
fix $X_0\in \sL^2(\sF_0^\gamma)$. 
For a finite horizon $T<\infty$, let 
$$
F:[0,T]\times \sD_n\times U \times\sP_2(\sD_n\times U)
 \to \R, \qquad G: \sD_n \times\sP_2(\sD_n\times U)\to \R,
$$ 
be given running and terminal cost functions satisfying 
Assumption \ref{asm:cost} below. 
Let $J(X_0,\alpha)$
be as in \eqref{eq:MFC_cost1}, and recall
that $\sG^\gamma$ is the completion of the $\sigma$-algebra
generated by the common noise processes $(W^{\gamma,0},N^{\gamma,0})$. 
In the absence of common noise, $\sG^\gamma$ is the trivial $\sigma$-algebra.
We have the following important technical observation.

\begin{Lem}
\label{lem:technical}
For any $\alpha\in\sA(\gamma)$ the $\sP_2(\sD_n \times U)$ measure-valued process
$$
\cL^{\gamma,\alpha}_t= \cL((X^{\gamma,\alpha}_{\cdot\land t},\alpha_t) \,|\, \sG^\gamma)
    = \cL((X^{\gamma,\alpha}_{\cdot\land t},\alpha_t) \, 
    | \, \sigma(W^{\gamma,0}_s,N^{\gamma,0}_s([0,s]\times B):s\leq t,B\in\sE))
$$
has an $\F^\gamma$-progressively measurable modification, and
we always consider this version.
\end{Lem}
\begin{proof}
Equipped with the product of the Skorokhod $J_1$ 
and Euclidean topologies,
$\sD_n \times U$  is a separable metric space. 
Since $\sF^\gamma_0$ is assumed to contain all 
$(\sF^\gamma,\P^\gamma)$-null sets, $\cL^{\gamma,\alpha}_t$
is adapted to the filtration $\F$.  Therefore, there exists a 
$\F^\gamma$-progressively measurable modification of $\cL^{\gamma,\alpha}$.
\end{proof}

For a given $X_0 \in \sL^2(\sF^\gamma_0)$, the
\emph{value function} of the MFC is defined by,
\be
\label{eq:MFC_cost}
v(X_0,\gamma):= \inf_{\alpha\in\sA(\gamma)}  J^\gamma(\alpha, X_0) 
=\inf_{\alpha \in \sA(\gamma)}  \E ^\gamma \left[\int_0^T 
F(t, X^{\alpha}_{\cdot\wedge t},\alpha_t,\cL^{\gamma,\alpha}_t)\,\d t 
+ G(X^{\alpha}_{\cdot\land T},\cL^{\gamma,\alpha}_T) \right]
\ee
where $X^\alpha$ solves \eqref{eq:state} on $\gamma$. Further, for $\varrho \in \sP_2(\R^n)$, let $\Gamma(\varrho)$ be the 
set of all bases $\gamma$ such that there exists $X_0 \in \sL^2(\sF_0^\gamma)$
with distribution $\varrho$, let $\cI^\gamma(\varrho)$
be the set of all such initial conditions, and we set
\be
\label{eq:V}
V(\varrho):= \inf_{\gamma \in \Gamma(\varrho)} \inf_{X_0 \in \cI^\gamma(\varrho)}
v(X_0,\gamma)
=\inf_{\gamma \in \Gamma(\varrho)} \inf_{X_0 \in \cI^\gamma(\varrho)} 
\inf_{\alpha\in\sA(\gamma)}  J^\gamma(\alpha, X_0).
\ee
Under mild conditions on the coefficients,
$v$ depends only on the law of $X_0$,
and we now formalize this property.

\begin{Def}
\label{def:invariance}
{\rm{We say  that an  MFC problem
is \emph{invariant under the probability bases}, 
for any $\gamma$,
$X_0\in \sL^2(\sF^\gamma_0)$ and $\varrho \in \sP_2(\R^n)$
we have,
$$
\cL^{\gamma}(X_0)= \varrho
\quad \Rightarrow \quad
v(X_0,\gamma) =  V(\varrho).
$$
}}
\end{Def}

In the diffusion setting,
invariance is proved in \cite[Theorem 3.6]{Cosso2023}
for `rich enough' probabilistic structures, see \cite[Section 2.1.3]{Cosso2023},
and in \cite[Proposition 2.4]{McKean_Vlasov_Formulations}, and our
Theorem~\ref{thm:invariance}
extends it to more general models.
Although almost all problems have this property,
some natural continuity and growth assumptions are needed
as summarized in Assumption~\ref{asm:invariance}} and  the counter
example provided in  \cite[Section 3]{Cosso2023}
highlights their necessity.

In what follows, we sometimes suppress the dependence on $\gamma$
and write $v(X_0)$, $J(X_0,\alpha)$.  Additionally, although the terminal cost $G$ depends 
on $\cL^\alpha_T$ only through its projection
$\cL(X^{\alpha}_{\cdot\land T} \,|\, \sG)$,
this observation is not relevant in our analysis and we do not emphasize it.

\subsection{The Potential Mean Field Game of Controls}
\label{sec:mfg}

In this subsection,
we give the definition of the \emph{potential mean field game 
of controls equilibrium under common noise};
see \cite[Definition 2.1]{carmona2016common} for MFGs under common noise and \cite{djete2022extended, djete2023mean} 
for general MFC and MFG of controls close to our setting. 

In these problems, the representative 
player takes an $\F$-adapted measure valued process
$t\mapsto\nu^*_t\in\sP_2(\sD_n \times U)$ as given 
and solves the optimal control problem  
$
\inf_{\alpha\in\sA} J_g(X_0,\alpha,\nu)
$
where $J_g$ is as in \eqref{eq:MFG-cost} with cost functions given by \eqref{eq:pot-game-cost}. 
Then, the Nash equilibrium is defined as follows.

\begin{Def}\label{def:MFG_equilibrium}
    {\rm{Consider  an initial condition $X_0\in \sL^2(\sF_0)$. A pair $(\alpha^*,\nu^*)$ 
    is a}} (potential) MFG of controls Nash equilibrium
    {\rm{if the followings hold:
    \begin{enumerate}[(i)]
        \item The control process $\alpha^* \in \sA$ satisfies
        $J_g(X_0,\alpha^*,\nu^*)  = \inf_{\alpha\in\sA}\
        J_g(X_0,\alpha,\nu^*)$.
        \item $\nu^* = (\nu^*_t)_{t\in[0,T]}$ is an 
        $\F$-progressive, $\sP_2( \sD_n\times U)$-valued process that satisfies
        \begin{equation}
        \label{eq:verify}
            \nu^*_t = \cL^{\alpha^*}_t=\cL((X^{\alpha^*}_{\cdot\land t}, \alpha^*_t) \,|\, \sG), 
            \qquad \mathrm{Leb}\otimes\P{\text{-a.e.}}
        \end{equation} 
    \end{enumerate}}}
\end{Def}

\begin{Rmk}{\rm{
In the absence of common noise, 
$\nu^* = (\nu^*_t)_{t\in[0,T]}$ is a deterministic 
$\sP_2(\sD_n\times U)$-valued 
process.  Then, the 
second condition above reduces to 
$\nu^*_t = \cL(X^{\alpha^*}_{\cdot\land t}, \alpha^*_t)$
for a.e.\ $t\in[0,T]$.
In this case, Borel measurability of 
$t\mapsto\cL(X^{\alpha^*}_{\cdot\land t}, \alpha^*_t)$ 
follows directly from the measurability of 
$(X^{\alpha^*}, \alpha^*)$. }}
\end{Rmk}

\subsection{Main Result}
\label{ssec:main}

Our main result provides a connection between the MFC and MFG 
problems.  It is the central result of the paper and 
its proof is given in Section \ref{sec:proof} 
by probabilistic arguments.  

To state our assumptions, for any $\nu \in \sP_2(\sD_n\times U)$, we denote the second moment by
\begin{equation}\label{eq:2nd_moment}
    \fs(\nu):=
\int_{\sD_n\times U}(\|x\|^2_T + |a|^2) \,\nu(\d x, \d a).
\end{equation}

\begin{Asm}\label{asm:cost}{\rm{
We assume that $F$,$G$ are measurable and for $(t,x,x',\nu,a,a') \in [0,T]\times \sD_n^2 \times
\sP_2(\sD_n\times U) \times U^2$
they satisfy:
\begin{itemize}
\item (Differentiability). The functions $F(t,x,a,\cdot)$ and $G(x,\cdot)$ are linearly differentiable with measurable the linear derivatives $\delta_\nu F$ and $\delta_\nu G$.
\item (Continuity). The maps $F(t,x,a,\cdot)$, $\pmu F(t,x,a,\cdot)(x',a')$, $G(x,\cdot)$ and $\pmu G(x,\cdot)(x')$ are continuous.
\item (Quadratic Growth).
\begin{align*}
    |F(t,x,a,\nu)| + |G(x,\nu)| + &\left| \pmu F(t,x,a,\nu)(x',a') \right| +\left| \pmu G(x,\nu)(x') \right|\\
    & \leq C(1+\|x\|_T^2 +|a|^2  +\fs(\nu)+ \|x'\|_T^2 + |a'|^2).
\end{align*}
\end{itemize}}}
\end{Asm}

We prove the main result under the assumption of invariance which is presented in Theorem~\ref{thm:invariance} under Assumption \ref{asm:invariance}.
Additionally, its proof reveals that it holds under a weaker version
of invariance as discussed in Remark \ref{rmk:condition} below.

\begin{Thm}\label{thm:pot_game_finite}
    Suppose that  Assumptions \ref{asm:state} and \ref{asm:cost} hold, and 
    the MFC problem  with payoff functional \eqref{eq:MFC_cost} 
    is invariant under probability bases.
    Then, for $X_0\in \sL^2(\sF_0)$, $\alpha^* \in \sA$,
    and $ \nu^*_t = \cL((X^{\alpha^*}_{\cdot\land t}, \alpha^*_t) \,|\, \sG)$,
    $$
     J(X_0,\alpha^*)= \inf_{\alpha \in \sA}\ J(X_0,\alpha)
     \quad \Rightarrow \quad
     J_g(X_0, \alpha^*, \nu^*)=
     \inf_{\alpha \in  \sA} \ J_g(X_0,\alpha, \nu^*).
     $$
    Consequently, for any minimizer  $\alpha^*\in\sA$ of the MFC  
    problem \eqref{eq:MFC_cost}, 
    the pair $(\alpha^*, \nu^*)$ is a Nash equilibrium for
    the  MFG of controls  \eqref{eq:MFG-cost}.
  \end{Thm}

As we discuss in Section \ref{sec:comparison_Cardaliaguet}, our above result is an extension of 
the connection proposed in \cite{briani2018stable}
and the references therein.
We additionally show that
when the MFC problem
has the widely employed 
separable structure, the resulting
MFG is a standard one which depends only on the distribution
of the state process and not the control.

\subsection{Law-Dependent Dynamics} \label{ssec:law-dep}
This section discusses a variation of our approach to MFC problems in which the controlled state dynamics are allowed to depend on the conditional law of the state process $\cL(X^\alpha_t\,|\,\sG)$. 
We again derive a mean field game such that minimizers of the control problem give rise to Nash equilibria of this game. 

We fix a probabilistic basis $\gamma$ and for given $X_0\in\sL^2(\sF_0)$, $\alpha\in\sA$, the state process $X^\alpha$ satisfies an equation of McKean-Vlasov type,
\begin{equation}
\begin{split}
X^\alpha_t = X_0 + \int_0^t &
b(s,X^\alpha_{\cdot\land s},\alpha_s,\mu^\alpha_s)\,\d s + \int_0^t 
\boldsymbol{\sigma}(s,X^\alpha_{\cdot\land s})\,\d \bW_s+ \int_0^t\int_E 
\boldsymbol{\lambda}(s,X^\alpha_{\cdot\land s-},\zeta)\, 
\wt\bN (\d s,\d \zeta),
\end{split}
\label{eq:state-ext} 
\end{equation}
where $\mu^\alpha_t= \cL(X^\alpha_{\cdot\land t} \,|\, \sG)$, $\boldsymbol{\sigma} = (\sigma^0,\sigma)$, and
$\boldsymbol{\lambda} = (\lambda^0,\lambda)$.  
Under usual Lipschitz and linear growth assumptions, this equation has a unique strong solution for any $\alpha\in\sA$, see Section \ref{ssec:proof-law-dep}. The mean field control problem is,
\be
\label{eq:MFC-cost-ext}
\inf_{\alpha \in \sA}  \ J(X_0,\alpha) \quad \text{subject to}\quad \eqref{eq:state-ext},
\ee
where $J(X_0,\alpha)$ is as in $\eqref{eq:MFC_cost}$.

The main assumption that enables us to associate a mean field game to this control problem is the following (left) invertibility condition on the drift: If $B\in\sB(\R^n)$ denotes the image of the drift $b$, then there exists a function $
\varphi:[0,T]\times\sD_n\times B \times \sP_2(\sD_n)\to U
$
such that $\varphi(t,x,b(t,x,a,\mu),\mu) = a$ holds for all $(t,x,a,\mu)\in[0,T]\times\sD_n\times B \times \sP_2(\sD_n)$. The precise assumptions are stated in Section \ref{ssec:proof-law-dep}, and a special case is presented in Section \ref{ssec:special-case-law}. Define running and terminal cost functions by,
\begin{equation*}
\label{eq:MFG-cost-functionals-ext}
\begin{split}
f(t,x,a,\nu) &= \int_{\sD_n\times U} \delta_\mu F(t,\wt x, \wt a, \mu)(x)\,\nu(\d \wt x,\d \wt a) + F(t,x,a,\mu) \\
&\qquad + \int_{\sD_n\times U} \nabla_a F(t,\wt x , \wt a,\mu) \cdot (\delta_\mu \varphi)(t,\wt x , b(t,\wt x , \wt a ,\mu), \mu)(x)\,\nu(\d \wt x, \d \wt a ),
\\
g(x,\mu)&= \int_{\sD_n} \delta_\mu G(\wt x,\mu)(x) \,\mu(\d \wt x) + G(x,\mu),
\end{split}
\end{equation*}
for $(t,x,a,\nu)\in[0,T]\times\sD_n\times U\times\sP_2(\sD_n\times U)$ and as usual, $\mu=\Pi_\sharp\nu$. We emphasize that 
$$
(\delta_\mu \varphi)(t,\wt x , b(t,\wt x , \wt a ,\mu), \mu)(x) = (\delta_\mu \varphi)(t,\wt x, y, \mu)(x)\,\Bigl|_{y=b(t,\tilde x , \tilde a ,\mu)}.
$$

For a given adapted measure flow $\mu=(\mu_t)_{t\in[0,T]}$ on $\sD_n$ and a control $\alpha\in\sA$, we define the state process $X^{\alpha,\mu}$ of a representative player of the mean field game by,
\begin{equation}
\label{eq:MFG-state-law-dep} 
X^{\alpha,\mu}_t = X_0 + \int_0^t 
b(s,X^{\alpha,\mu}_{\cdot\land s},\alpha_s,\mu_s)\,\d s + \int_0^t 
\boldsymbol{\sigma}(s,X^{\alpha,\mu}_{\cdot\land s})\,\d \bW_s +\int_0^t\int_E 
\boldsymbol{\lambda}(s,X^{\alpha,\mu}_{\cdot\land s-},\zeta)\, 
\wt\bN (\d s,\d \zeta).
\end{equation}

The definition of a mean field game equilibrium reads analogous to Definition \ref{def:MFG_equilibrium}. We say that a pair $(\alpha^*,\nu^*)$ is a \emph{Nash equilibrium of the corresponding MFG of controls} if the following hold:
\begin{enumerate}[(i)]
\item The control $\alpha^*\in\sA$ satisfies $ J_g(X_0,\alpha^*,\nu^*)  = \inf_{\alpha\in\sA}\
 J_g(X_0,\alpha,\nu^*)$ where, with an abuse of notation,
\begin{equation}\label{eq:MFG-cost-ext}
 J_g(X_0,\alpha,\nu) = \E
\left[\int_0^T f(t,X^{\alpha,\mu}_{\cdot\land t}, \alpha_t, \nu_t)\,\d t 
+ g(X^{\alpha,\mu}_{\cdot\land t}, \mu_T) \right].
\end{equation}
Here, $(f,g)$ are as in \eqref{eq:MFG-cost-functionals-ext} and $X^{\alpha,\mu}$ solves \eqref{eq:MFG-state-law-dep} with $\mu=\Pi_\sharp\nu$.

\item $\nu^* = (\nu^*_t)_{t\in[0,T]}$ is an 
$\F$-progressive, $\sP_2( \sD_n\times U)$-valued process that satisfies
$$
    \nu^*_t = \cL((X^{\alpha^*,\nu^*}_{\cdot\land t}, \alpha^*_t) \,|\, \sG), 
    \qquad \mathrm{Leb}\otimes\Prob{\text{-a.e.}}
$$
\end{enumerate}

In section \ref{ssec:proof-law-dep} we prove the following result.

\begin{Thm}\label{thm:pot-game-law-dep}
Suppose that  Assumption \ref{asm:law-dep}  holds and that 
the MFC problem \eqref{eq:MFC-cost-ext} is invariant under probability bases.
Then, for $X_0\in \sL^2(\sF_0)$, $\alpha^* \in \sA$,
and $ \nu^*_t = \cL((X^{\alpha^*}_{\cdot\land t}, \alpha^*_t) \,|\, \sG)$,
$$
J(X_0,\alpha^*)= \inf_{\alpha \in \sA}\ J(X_0,\alpha)
 \quad \Rightarrow \quad
J_g(X_0, \alpha^*, \nu^*)=
 \inf_{\alpha \in  \sA}\ J_g(X_0,\alpha, \nu^*),
$$
where $J$ is as in \eqref{eq:MFC-cost-ext}, and $J_g$ is as in \eqref{eq:MFG-cost-ext}.
\end{Thm}

\begin{Rmk}
{\rm An inspection of our proof in Section \ref{ssec:proof-law-dep} shows that the dependence on $(\alpha,\mu)$ could be through any of the coefficients $\sigma,\sigma^0,\lambda,\lambda^0$ instead of the drift $b$ as in \eqref{eq:state-ext}. Then, an analogous result to Theorem \ref{thm:pot-game-law-dep} holds provided the controlled coefficient satisfies an invertibility condition, and the other coefficients are independent of $(\alpha,\mu)$. 
}
\end{Rmk}

\section{Proof of Theorem \ref{thm:pot_game_finite}}
\label{sec:proof}

The main idea is to compose the optimal control $\alpha^*$
with an arbitrary control $\alpha$ through a biased independent coin-flip.
To achieve this, we extend the original probabilistic structure by adding an initial Bernoulli
random variable, and then we use either one of the controls
depending on its outcome.  This provides linearity
in the distribution variables and we obtain the result
by essentially differentiating with respect to the success probability
of the coin-flip.  The assumed invariance principle
ensures that $\alpha^*$ is still a minimizer
of the extended problem.

\subsection{Preliminaries}
\label{ssec:pre}

We first extend the probability basis $\gamma$ 
to $\hat \gamma$ as follows,
\begin{equation*}
    \hat \Omega := \Omega \times \{0,1\}, \quad  
    \hat\sF := \sF \otimes \sF^\diamond, \quad 
    \hat \P := \P \otimes \tfrac12 (\delta_{\{0\}}+\delta_{\{1\}}),\quad  
    \hat \sF_t :=\sF_t \otimes \sF^\diamond,
\end{equation*}
where $\sF^\diamond$ is the power set of $\{0,1\}$.
We naturally extend any function $\zeta$ on 
$\Omega$ to a function on $\hat\Omega$ by setting 
$\hat  \zeta (\omega,x):= \zeta (\omega)$, 
$(\omega,x) \in \hat \Omega$.
Then, $\hat \gamma := (\hat \Om, \hat \sF, \hat \P,
\hat \F, \hat \sW, \hat \sN)$ is a probability basis.  
Additionally, the $\sigma$-algebra
$\hat \sG$ generated by the common noise 
$(\hat W^0, \hat N^0)$ is given by
$$
\hat \sG = \sG \otimes \sG^\diamond,
$$
where $\sG^\diamond:=\{\emptyset,\{0,1\}\}$ is the trivial $\sigma$-algebra.
Given  any two functions $\zeta,\eta$  on $\Om$, 
we define a new function $\zeta \otimes \eta$ on $\hat \Om$ by,
$$
(\zeta \otimes \eta)(\om,x):= \zeta(\om) \chi_{\{x=0\}}+
\eta(\om) \chi_{\{x=1\}},
\qquad (\omega,x) \in \hat \Om.
$$
Then, for any $Z=Z(\omega,x)$ we have
 $Z=Z^{(0)} \otimes Z^{(1)}$ where $Z^{(x)}(\om)=Z(\om,x)$.

Given $\delta \in (0,1)$, we define a probability measure
$\Pd$ equivalent to $\hat \P$ by,
\begin{equation}
\label{eq:pdelta}
\Ed [Z] = \E^{\Pd}[Z]:= \delta \, \E[Z^{(0)}] + 
(1-\delta) \, \E[Z^{(1)}].
\end{equation}
This defines a family of probability bases $\gamma^\delta = (\hat \Om, \hat \sF, \P^\delta,
\hat \F, \hat \sW, \hat \sN)$ indexed by $\delta \in (0,1)$. We note that
$\hat \sA := \sA(\hat \gamma) = \sA(\gamma^\delta)$.
For $\beta \in \hat \sA$, the projected controls
$\beta^{(0)}, \beta^{(1)}$ are in $\sA$ and by definition 
$\beta_t =  \beta^{(0)}_t \otimes \beta^{(1)}_t$.
Let $Y^\beta$ be the solution to 
$\mathrm{SDE}^{\hat\gamma}(\hat X_0, \beta)$, and $X^{(0)}:= X^{\beta^{(0)}}$,
$X^{(1)}:= X^{\beta^{(1)}}$ be the controlled state processes
solving the equation \ref{eq:state} with controls
$\beta^{(0)}$ and $\beta^{(1)}$, respectively.
Further, we set,
\begin{equation*}
    \nu^{\beta,\delta}_t 
    := \cL^{\Pd}((Y^\beta_{\cdot\land t}, \beta_t)\,|\, \hat \sG\, ),\qquad
    \nu^{(x)}_t := \cL^{\P}((X^{(x)}_{\cdot\land t}, \beta^{(x)}_t)\,|\, \sG\, ),\qquad x=0,1.
\end{equation*}
Since $\hat \sG=\sG \otimes \sG^\diamond$, 
$\nu^{\beta,\delta}$ is identified with a function on $\Om$.
 
We have the following immediate consequences of the definitions.

\begin{Lem} 
\label{lem:additive}
For all $t \in [0,T]$, a.s., 
$Y_t^\beta= X_t^{(0)} \otimes X_t^{(1)}$
and $\nu^{\beta,\delta}_t = \delta \nu^{(0)}_t 
+ (1-\delta)  \nu^{(1)}_t$.
\end{Lem}

\begin{proof}
First consider the It\^{o}-integral,
\begin{align*}
I_t(\beta)(\om,x) &:= (\int_0^t \sigma(s,Y^\beta_{\cdot\land s},  \beta_s) \, \d \hat W_s)(\om,x)\\
&= (\int_0^t \sigma(s,X^{(0)}_{\cdot\land s},  \alpha^{(0)}_s)\chi_{\{x=0\}} \, \d \hat W_s)(\om,x)
+ (\int_0^t \sigma(s,X^{(1)}_{\cdot\land s},  \alpha^{(1)}_s) \chi_{\{x=1\}} \, \d \hat W_s)(\om,x)\\
&= (\int_0^t \sigma(s,X^{(0)}_{\cdot\land s},  \alpha^{(0)}_s) \, \d W_s)(\om) \chi_{\{x=0\}}
+ (\int_0^t \sigma(s,X^{(1)}_{\cdot\land s},  \alpha^{(1)}_s) \, \d W_s)(\om) \chi_{\{x=1\}}\\
&=: I_t(\alpha^{(0)})(\om)\chi_{\{x=0\}} +I_t(\alpha^{(1)})(\om)\chi_{\{x=1\}},\quad \text{a.s.},
\end{align*}
where the final equality follows from the construction of the stochastic integral.
Hence, $I_t (\beta)= I_t (\alpha^{(0)}) \otimes I_t (\alpha^{(1)})$ a.s.  A similar computation of all other four integrals in 
$\mathrm{SDE}^{\hat\gamma}(\hat X_0, \beta)$ implies
that $X^{(0)} \otimes X^{(1)}$ is a solution
of  $\mathrm{SDE}^{\hat\gamma}(\hat X_0, \beta)$. 
Then by strong uniqueness, we conclude that $Y^\beta=X^{(0)} \otimes X^{(1)}$.

Let $\xi :(\om,x)\in\hat \Om \to \xi(\om,x)=x$ be the canonical projection.
For $t\in[0,T]$ and a Borel set $S\subset \sD_n\times U$, 
the followings hold a.s.,
\begin{align*}
\nu^{\beta,\delta}_t(S) &= \P^\delta((Y^\beta_{\cdot\land t},\beta_t)\in S\,|\,
\hat \sG ) \\
&= \Pd((Y^\beta_{\cdot\land t},\beta_t)\in S 
\ \text{and}\ \xi =0
\,|\, \hat \sG)+
 \Pd((Y^\beta_{\cdot\land t},\beta_t)\in S 
\ \text{and}\  \xi =1 \,|\, \hat \sG)\\
&= \Pd((X^{(0)}_{\cdot\land t},\alpha^{(0)}_t)\in S 
\ \text{and}\  \xi =0
\,|\, \hat \sG)+
\Pd ((X^{(1)}_{\cdot\land t},\alpha^{(1)}_t)\in S 
\ \text{and}\  \xi =1
\,|\, \hat \sG).
\end{align*}
Since the canonical map $\xi$ is independent
of all random elements defined on 
the original probability space $\Om$
and $\hat \sG = \sG \otimes \sG^\diamond$, a.s.,
\begin{align*}
    \nu^{\beta,\delta}_t(S)(\om,x) &=
    \Pd(\xi=0)\ \P((X^{(0)}_{\cdot\land t},\alpha^{(0)}_t)\in S 
    \,|\, \sG ) +
    \Pd(\xi=1)\ \P((X^{(1)}_{\cdot\land t},\alpha^{(1)}_t)\in S 
    \,|\, \sG) \\
    &= 
    \delta \, \nu^{(0)}_t(S)(\om) +
    (1-\delta)\,  \nu^{(1)}_t(S)(\om).
\end{align*}
\end{proof}
For a given $Y_0\in \sL^2(\hat \sF_0)$, $\delta \in (0,1)$,
and a $\sP_2(\sD_n\times U)$-valued process $\eta$ on $\Om$, set
\begin{align*}
\cJ_\delta(Y_0,\beta,\eta)&:= 
\E^\delta \left[ \int_0^T F(t,Y^\beta_{\cdot\wedge t},\beta_t, \hat \eta_t)\ \d t
+ G(Y^\beta_{\cdot\wedge T},\hat \eta_T) \right],\\
\cJ(X_0,\alpha,\eta)&:= 
\E \left[ \int_0^T F(t,X^\alpha_{\cdot\wedge t},\alpha_t,  \eta_t)\ \d t
+ G(X^\alpha_{\cdot\wedge T}, \eta_T) \right].
\end{align*}
Then, the MFC problem on $\gamma^\delta$ is to minimize
$J_\delta(Y_0,\beta) := \cJ_\delta (Y_0,\beta,\nu^{\beta,\delta})$
over all $\beta \in \hat \sA$.
Moreover, in view of \eqref{eq:pdelta},
\begin{equation}
\label{eq:deltaJ}
J_\delta(Y_0,\beta) = \delta \ \cJ(Y_0^{(0)}, \beta^{(0)}, \nu^{\beta,\delta})
+ (1-\delta) \ \cJ(Y_0^{(1)}, \beta^{(1)}, \nu^{\beta,\delta}).\,
\end{equation}
Additionally, as $\cL^{\P^\delta}(\hat X_0) = \cL^\P(X_0)$,
the assumed invariance principle implies that
\begin{equation}
\label{eq:invariance}
\inf_{\beta \in \hat \sA} J_\delta(\hat X_0 ,\beta)
= \inf_{\alpha \in \sA} J(X_0 , \alpha).
\end{equation}

\subsection{Completion of the Proof}
\label{ssec:conclude}

Let $\alpha^*\in\sA$ be a minimizer
of MFC \eqref{eq:MFC_cost} starting from
initial condition $X_0\in \sL^2(\sF_0)$, and 
let $\alpha\in\sA$ be an arbitrary control process. 
To simplify the presentation, we write
 $X^*=X^{\alpha^*}$ for the optimal state process
 and 
 $$
 \nu^\alpha_t:= \cL^{\alpha}_t=\cL((X^\alpha_{\cdot \wedge t},\alpha_t)\, |\, \sG),\qquad
  \nu^*_t:= \cL^{\alpha^*}_t=\cL((X^*_{\cdot \wedge t},\alpha^*_t)\, |\, \sG).
  $$
Set  $\beta := \alpha \otimes \alpha^*$,
$Y_0= X_0 \otimes X_0 =\hat X_0$, so that
$\beta^{(0)}= \alpha$, $\beta^{(1)}=\alpha^*$,
$Y_0^{(0)}=Y_0^{(1)}=X_0$, and by Lemma \ref{lem:additive},
$ \nu^{\beta,\delta}_t = \delta \nu^\alpha_t +(1-\delta)\nu^*_t$.
Moreover, by \eqref{eq:invariance}, 
$J(X_0,\alpha^*) \le J_\delta(Y_0,\beta)=J_\delta(\hat X_0,\beta)$.

 The following  calculation is central to the proof.
\begin{Lem}
\label{lem:pot_game_finite}
Let $J_g$ be as in \eqref{eq:MFG-cost}. Then,
\begin{equation*}
   0\le  \lim_{\delta\downarrow0} \frac{1}{\delta}
    \left(J_\delta(\hat X_0,\beta) - J(X_0,\alpha^*)\right)
    = J_g(X_0,\alpha,\nu^*) - J_g(X_0,\alpha^*,\nu^*).
\end{equation*}
In particular, $\alpha^*$ minimizes $J_g(X_0,\alpha,\nu^*)$ over all $\alpha \in \sA$.
\end{Lem}

\begin{proof} 
We first note that the growth conditions in Assumption \ref{asm:cost} and the second moment estimate of the controlled state process in Theorem \ref{thm:state} justify the use of Fubini's theorem in the rest of this proof. 

By the definitions of $J_\delta, \cJ_\delta$, and $\cJ$,
\begin{align*}
J_\delta(\hat X_0,\beta) 
& =  \cJ_\delta(\hat X_0,\beta, \nu^{\beta,\delta}) 
= \delta \cJ(X_0,\alpha,\nu^{\beta,\delta}) +
 (1-\delta) \cJ(X_0,\alpha^*,\nu^{\beta,\delta}).
\end{align*}
Then, by   \eqref{eq:deltaJ},
$$
 \frac{1}{\delta}
    \left(J_\delta(\hat X_0,\beta) - J(X_0,\alpha^*)\right) =
    \cI_{1,\delta} +  \cI_{2,\delta},\quad \text{where}
$$
$$
  \cI_{1,\delta} :=     \cJ(X_0,\alpha,\nu^{\beta,\delta}) 
  -  \cJ(X_0,\alpha^*,\nu^{\beta,\delta}),\quad
    \cI_{2,\delta} :=    \frac{1}{\delta}
    \left(  \cJ(X_0,\alpha^*,\nu^{\beta,\delta}) -  J(X_0,\alpha^*)\right).
$$
Since $\nu^{\beta,\delta}$ converges to $\nu^*$ as $\delta$ tends to zero,
\begin{align}
\label{eq:1}
\lim_{\delta \downarrow 0}\ \cI_{1,\delta}  
&=   \cJ(X_0,\alpha,\nu^*) -  \cJ(X_0,\alpha^*,\nu^*)\\
\nonumber &= \E\left[\int_0^T \left[F(t,X^\alpha_{\cdot \wedge t},\alpha_t,\nu^{*}_t)
- F(t,X^*_{\cdot \wedge t},\alpha^*_t,\nu^{*}_t)\right] \d t 
+ G(X^\alpha_{\cdot \wedge T},\nu^{*}_T)
- G(X^*_{\cdot \wedge T}, \nu^{*}_T) \right].
\end{align}
To analyze the limit of $\cI_{2,\delta}$ we first observe that by  Lemma \ref{lem:additive}, a.s.,
\begin{align*}          
\lim_{\delta \downarrow 0} 
 \frac{1}{\delta} [F(t, &\,X^*_{\cdot \wedge t}, \alpha^*_t, \nu^{\beta,\delta}_t) 
    - F(t, X^*_{\cdot \wedge t}, \alpha^*_t,\nu^*_t)]\\
    =& \lim_{\delta \downarrow 0} \ 
      \int_0^1 \int_{\sD_n \times U} \pmu F(t,X^*_{\cdot \wedge t}, \alpha^*_t,
    \tau \delta \nu^\alpha_t+(1-\tau \delta)\nu^*_t)(x, a) \, (\nu^\alpha_t-\nu^*_t)(\d x, \d a) \ \d \tau\\
    = & \int_{\sD_n \times U} \pmu F(t,X^*_{\cdot \wedge t}, \alpha^*_t,
  \nu^*_t)(x,a) \, (\nu^\alpha_t-\nu^*_t)(\d x, \d a) .
\end{align*}
We use the definition $\nu^*_t = \cL((X^*_{\cdot \wedge t},\alpha^*_t)\, |\, \sG )$ 
and well-known properties of the conditional expectations, to compute:
\begin{align*}
    &\E\left[\int_{\sD_n\times U} 
    \pmu F(t, X^*_{\cdot\land t}, \alpha^*_t, \nu^*_t)(x,a)\, \nu^*_t(\d x, \d a) \right] \\
    &\hspace{50pt}=  \E\left[\E\left[\int_{ \sD_n\times U} 
    \pmu F (t, X^*_{\cdot\land t}, \alpha^*_t, \nu^*_t)(x,a)\, \nu^*_t(\d x, \d a) \, \Bigl| \, \sG\,  \right]\right] \\
     &\hspace{50pt}= 
     \E\left[\int_{ \sD_n\times U} \int_{ \sD_n\times U}
     	\pmu F (t,\tilde{x}, \tilde{a},\nu^*_t)(x,a)\, \nu^*_t(\d x, \d a)\,\nu^*_t(\d \tilde{x},\d \tilde{a}) \right] \\
     &\hspace{50pt}=\E\left[\int_{ \sD_n\times U} \int_{ \sD_n\times U} 
     	\pmu F (t,\tilde{x},\tilde{a}, \nu^*_t)(x,a)\, \nu^*_t(\d \tilde{x},\d \tilde{a})\, \nu^*_t(\d x, \d a) \right] \\
     &\hspace{50pt}= \E\left[\E\left[\int_{ \sD_n\times U} 
     	\pmu F (t,\tilde{x}, \tilde{a},\nu^*_t)(X^*_{\cdot\land t},\alpha^*_t)\, \nu^*_t(\d \tilde{x},\d \tilde{a}) \, \Bigl| \, \sG\, \right] \right] \\
     &\hspace{50pt}= \E\left[\int_{ \sD_n\times U} \pmu 
     F (t,\tilde{x},  \tilde{a}, \nu^*_t)(X^*_{\cdot\land t},\alpha^*_t)\, \nu^*_t(\d \tilde{x},\d  \tilde{a}) \right].
\end{align*}
As $\nu^\alpha_t = \cL((X^\alpha_{\cdot \wedge t},\alpha_t)\, |\, \sG)$,
a computation as above implies that
\begin{align*}
    &\E\left[\int_{\sD_n\times U} 
    \pmu F(t, X^*_t, \alpha^*_t, \nu^*_t)(x,a)\, \nu^\alpha_t(\d x,\d a) \right] \\
    &\hspace{50pt}= 
    \E\left[\int_{ \sD_n\times U} 
    \pmu F (t,\tilde x,  \tilde{a}, \nu^*_t)(X^\alpha_{\cdot\land t},\alpha_t)\, 
    \nu^*_t(\d \tilde x,\d  \tilde{a}) \right].
\end{align*}
Above calculations,  \eqref{eq:1}
and the definition of $f$ in \eqref{eq:pot-game-cost} imply that
\begin{align*}       
 \lim_{\delta \downarrow 0} \ \cI_{1,\delta}+
 \frac{1}{\delta}\ &\E\int_0^T \left[F(t,X^*_{\cdot \wedge t},\alpha^*_t,\nu^{\beta,\delta}_t)
- F(t,X^*_{\cdot \wedge t},\alpha^*_t,\nu^*_t)\right] \d t \\
&\hspace{30pt} =  \E\int_0^T \left[f(t,X^\alpha_{\cdot \wedge t},\alpha_t,\nu^*_t)
- f(t,X^*_{\cdot \wedge t},\alpha^*_t,\nu^*_t)\right] \d t .
\end{align*}
Similarly, with $g$ as in \eqref{eq:pot-game-cost},
$$
\lim_{\delta \downarrow 0}  \ \E\left[G(X^\alpha_{\cdot \wedge T}, \nu^{\beta,\delta}_T)
    - G(X^*_{\cdot \wedge T},\nu^{\beta,\delta}_T)\right]
    =\E\left[ g(X^\alpha_{\cdot \wedge T}, \nu^*_T)- g(X^*_{\cdot \wedge T}, \nu^*_T)\right].
$$
\end{proof}
We have shown that $\alpha^*$ minimizes $J_g(X_0,\alpha,\nu^*)$ and by its
definition $\nu^*_t=\cL((X^*_{\wedge t},\alpha^*_t)\, |\, \sG)$.
Hence, $(\alpha^*,\nu^*)$ is a Nash equilibrium of the MFG of controls.

\hfill $\Box$

\begin{Rmk}
\label{rmk:condition}
{\rm{We use the invariance property only in \eqref{eq:invariance}.
Therefore it suffices to require that the value function on $\gamma$
and its extension to $\hat \gamma$ coincide.
In fact, if the probability basis $\gamma$ is rich 
enough in the sense defined in \cite[Section 2.1.3]{Cosso2023}
and $X_0$ is deterministic,
one can directly prove that the extension to $\hat \gamma$
does not change the value.  

Additionally, if the triplet $(\gamma, X_0,\alpha^*)$
minimizes the `global' MFC problem \eqref{eq:V},
then the above proof implies that $(\alpha^*,\nu^*)$
is a Nash equilibrium on the basis $\gamma$ starting
from $X_0$.}}
\end{Rmk}

\subsection{The Case of Law-Dependent Dynamics}\label{ssec:proof-law-dep}

We assume that the coefficients 
\begin{align*}
b&:[0,T]\times \sD_n\times U \times \sP_2(\sD_n)\to \R^n,\\
(\sigma,\sigma^0)&:[0,T]\times \sD_n\to \sM_{n\times d} \times \sM_{n\times \ell},\\
(\lambda,\lambda^0)&:[0,T]\times \sD_n\times E \to \R^n\times \R^n
\end{align*}
satisfy the following assumption.

\begin{Asm} \label{asm:law-dep}{\rm
$F$ and $G$ satisfy Assumption \ref{asm:cost}. 
The following hold for all $(t,x,x',a, a', \mu, \mu')\in [0,T]\times \sD_n^2\times U^2 \times \sP_2(\sD_n)^2$:
\begin{enumerate}[(i)]

\item (Lipschitz and Linear Growth). The coefficients  $\sigma,\sigma^0,\lambda,\lambda^0$ satisfy Assumption \ref{asm:state}. The drift coefficient $b$ is measurable and satisfies 
\begin{align*}
    |b(t,x,a,\mu)-b(t,x,a,\mu')|&\leq L(\|x-x'\|_t+\fm_2(\mu,\mu')),\\
    |b(t,x,a,\mu)|^2&\leq K (1+\|x\|_t^2  + |a|^2 + \fs(\mu)).
\end{align*}
\item (Left Inverse). There exists a Borel set $B\subset \R^n$ and a measurable function 
$
\varphi:[0,T]\times\sD_n\times B \times \sP_2(\sD_n)\to U
$
such that
$$
\quad \varphi(t,x,b(t,x,a,\mu),\mu) = a.
$$

\item (Differentiability). There exist functions $\nabla_a F$ and $\delta_\mu F$ with 
\begin{align*}
F(t, x ,a,\mu) - F(t, x,a', \mu')  &= 
\int_0^1 \nabla_a F(t,x,a' + \tau (a-a'), \mu') \cdot (a-a')\,\d \tau  \\
&\qquad + \int_0^1\int_{\sD_n} \delta_\mu F(t,x, a', \mu' + \tau (\mu-\mu'))(x)\,(\mu-\mu')(\d x)\,\d \tau.
\end{align*}
Furthermore, $\varphi(t,x,y,\cdot)$ is linearly differentiable a with measurable derivative $\delta_\mu \varphi$. 

\item (Continuity and Growth). The functions $\delta_\mu\varphi(t,x,y,\cdot)(x')$ and $F(t,x,\cdot,\cdot)$ are continuous and satisfy
$$
| \nabla_a F(t,x,a,\mu) | + |\varphi(t,x,y,\mu)|^2 + |\delta_\mu\varphi(t,x,y,\mu)(x')|^2 \leq C(1+\|x\|_t^2+\|x'\|_t^2+|a|^2+|y|^2+\mathfrak{s}(\mu)).
$$
\end{enumerate}}
\end{Asm}

Now fix a probabilistic basis $\gamma$. Assumption \ref{asm:law-dep} in particular guarantees strong existence and uniqueness \ref{eq:state-ext} for any $X_0\in\sL^2(\sF^\gamma_0)$ and $\alpha\in\sA(\gamma)$. Condition (iii) states that the running cost is jointly differentiable in $(a,\mu)$ while condition (ii) is an left invertibility condition on the drift $a\mapsto b(t,x,a,\mu)$. \\

We now prove Theorem \ref{thm:pot-game-law-dep}, and we use the notation and setting introduced in Section \ref{ssec:pre}. 
We fix an arbitrary control $\alpha\in\sA(\gamma)$ and an optimal control $\alpha^*\in\sA(\gamma)$ for the problem \eqref{eq:MFC-cost-ext}. We write
$$
X^* := X^{\alpha^*},\quad \mu^*_t := \cL^{\P}(X^*_{\cdot\land t}\,|\,\sG), 
\quad X := X^{\alpha,\mu^*} , 
\quad \mu_t := \cL^{\P}(X_{\cdot\land t}\,|\,\sG),
$$ 
along with $\mu^\delta := \delta \mu + (1-\delta) \mu^*$, $\delta\in(0,1)$. 
Here, $X^{\alpha,\mu^*}$ refers to the solution of \eqref{eq:MFG-state-law-dep}. 
Define a control $\beta^\delta$ on $\hat\Omega$ by,
$$
\beta^\delta _t := \varphi(t,X_{\cdot\land t}\otimes X^*_{\cdot\land t}, b(t,X_{\cdot\land t},\alpha_t,\mu^*_t) \otimes b(t,X^*_{\cdot\land t},\alpha^*_t,\mu^*_t), \mu^\delta_t).
$$
The growth assumptions on $\varphi$ ensure that $\beta^\delta$ is an admissible control. On the basis $\gamma^\delta$, let $Y^\delta$ denote the solution to \eqref{eq:state-ext} with initial condition $\hat X_0$ and control $\beta^\delta$. With this, we have the following analogue of Lemma \ref{lem:additive}.

\begin{Lem} 
\label{lem:additive-ext}
Fix $\delta\in(0,1)$. On $\gamma^\delta$, we have
$Y^\delta = X\otimes X^*$
and $\mu^\delta_t = \cL^{\P^\delta}(Y^\delta_{\cdot\land t}\,|\,\hat\sG) $.
\end{Lem}
\begin{proof}
We check that $X\otimes X^*$ solves  \eqref{eq:state-ext} with initial condition $\hat X_0$ and control $\beta^\delta$.
First, note $(X\otimes X^*)_0=\hat X_0$. Further, by definition of $\varphi$ as an inverse,
\begin{align*}
\int_0^t b(t,X_{\cdot\land s}\otimes X^*_{\cdot\land s},\beta^\delta_s,\mu^\delta_s)\,\d s  &= \int_0^t [b(s,X_{\cdot\land s},\alpha_s,\mu^*_s) \otimes b(s,X^*_{\cdot\land s},\alpha^*_s,\mu^*_s)]\,\d s\\
&= \left(\int_0^t b(s,X_{\cdot\land s},\alpha_s,\mu^*_s) \,\d s \right) \otimes \left(\int_0^t  b(s,X^*_{\cdot\land s},\alpha^*_s,\mu^*_s)\,\d s\right).
\end{align*}
Since the remaining dynamics do not depend on $\alpha$ or the mean field interaction, this proves $\mu^\delta_t = \cL^{\P^\delta}((X\otimes X^*)_{\cdot\land t}\,|\,\hat\sG)$. This further implies that $X\otimes X^*$ satisfies \eqref{eq:state-ext} with control $\beta^\delta$ on the basis $\gamma^\delta$. By strong uniqueness, we deduce $Y^\delta = X\otimes X^*$ on $\gamma^\delta$.
\end{proof}

In analogy to our proof presented in Section \ref{ssec:pre}, for $\delta \in (0,1)$
and a $\sP_2(\sD_n)$-valued process $\eta$ on $\Om$, set
\begin{align*}
\cJ_\delta(\eta)&:= 
\E^\delta \left[ \int_0^T F(t,Y^\delta_{\cdot\wedge t},\beta^\delta_t, \hat \eta_t)\ \d t
+ G(Y^\delta_{\cdot\wedge T},\hat \eta_T) \right],\\
\cJ(\alpha,\eta)&:=
\E\left[ \int_0^T F(t,X^\alpha_{\cdot\wedge t}, \varphi(t,X^\alpha_{\cdot\land t}, b(t,X^\alpha_{\cdot\land t},\alpha_t,\mu^*_t), \eta_t), \eta_t)\,\d t + G(X^\alpha_{\cdot\wedge T},\eta_T) \right].
\end{align*}
where $X^\alpha$ solves \eqref{eq:state-ext} on $\gamma$.
The preceding lemma implies,
$$
\cJ_\delta (\mu^\delta) = \delta \, \cJ(\alpha,\mu^\delta) + (1-\delta) \, \cJ(\alpha^*,\mu^\delta).
$$
Then,
$$
\frac{1}{\delta} \, ( \cJ_\delta (\mu^\delta) -  J(\alpha^*)) =  \cJ(\alpha,\mu^\delta) -   \cJ(\alpha^*,\mu^\delta) + \frac{1}{\delta} ( \cJ(\alpha^*,\mu^\delta) -  J(\alpha^*)) =: \cI_{1,\delta} + \cI_{2,\delta},
$$
where $J(\alpha^*)= J(X_0,\alpha^*)$ is as in \eqref{eq:MFC-cost-ext}. By the continuity and growth assumptions \ref{asm:law-dep},
$$
\lim_{\delta\downarrow0} \cI_{1,\delta} = \E
\left[\int_0^T [F(t,X_{\cdot\land t}, \alpha_t, \mu^*_t) - F(t,X^*_{\cdot\land t}, \alpha^*_t, \mu^*_t)]\,\d t 
+ G(X_{\cdot\land T}, \mu^*_T) - G(X^*_{\cdot\land T}, \mu^*_T) \right].
$$
For the second term, we note that $\varphi(t,X^*_{\cdot\land t}, b(t,X^*_{\cdot\land t},\alpha^*_t,\mu^*_t), \mu^*_t) = \alpha^*_t$ and use the chain rule to compute,
\begin{align*}
\lim_{\delta\downarrow 0} \ &\frac{1}{\delta} (F(t,X^*_{\cdot\wedge t}, \varphi(t,X^*_{\cdot\land t}, b(t,X^*_{\cdot\land t},\alpha^*_t,\mu^*_t), \mu^\delta_t), \mu^\delta_t) - F(t,X^*_{\cdot\wedge t}, \alpha^*_t, \mu^*_t)) \\
&= \sum_{j=1}^k \frac{\partial}{\partial a_j} F(t,X^*_{\cdot\wedge t},\alpha^*_t,\mu^*_t) \int_{\sD_n} (\delta_\mu \varphi_j)(t,X^*_{\cdot\land t}, b(t,X^*_{\cdot\land t},\alpha^*_t,\mu^*_t), \mu^*_t)(x)\,[\mu_t-\mu^*_t](\d x)\\
&\qquad + \int_{\sD_n} \delta_\mu F(t,X^*_{\cdot\wedge t}, \alpha^*_t, \mu^*_t)(x)\,[\mu_t-\mu^*_t](\d x).
\end{align*}
A similar calculation holds for the terminal cost $G$. We again appeal to Fubini's theorem to obtain,
$$
\lim_{\delta\downarrow0}\frac{1}{\delta} \, ( \cJ_\delta( \mu^\delta) - J(\alpha^*)) =  J_g(X_0,\alpha,\mu^*) - J_g(X_0,\alpha^*,\mu^*),
$$
where $J_g$ is as in \eqref{eq:MFG-cost-ext}.
If the problem is invariant under probability bases, then this limit is non-negative which completes the proof.

\section{Discussion and Special Cases}
\label{sec:discuss}

This section discusses special cases
that are frequently used in the literature and assumes the absence of 
common noise $(W^0,N^0)$. This assumption turns $t\mapsto\cL^\alpha_t$ into a deterministic 
flow, hence Nash equilibria only involve deterministic processes
$t \mapsto \nu_t \in \sP_2(\sD_n \times U)$.

\subsection{Separable Running Cost} 
\label{ssect:seperable}
We first investigate separable cost functions and state processes that follow \ref{eq:state}. Assume that for suitable functions,
$$
F_0: [0,T] \times \sD_n \times \sP_2(\sD_n) \to \R,\ \ 
F_1 :[0,T] \times \sP_2(\sD_n) \to \R, \ \
G_0 :\sP_2(\sD_n) \to \R, 
$$
and for $(t,x,\nu,a) \in [0,T]\times \sD_n
\times\sP_2(\sD_n\times U)\times U$ we have,
\begin{equation}\label{eq:MFC_cost_(a)}
    F(t,x,a,\nu) = F_0(t,x,a) + F_1(t,\mu), 
    \quad G(x,\nu) = G_0(\mu),
\end{equation}
where $\mu=\Pi_\sharp\nu$ and, as before, $\Pi: \sD_n \times U \to \sD_n$ is the projection on the first variable.
This setting corresponds to a
standard mean field game,
and in most applications $F_0= \tfrac12 |a|^2$.
Then, the costs
$f(t,x,\mu,a)$, $g(t,\nu)$ given in \eqref{eq:pot-game-cost} depend
only on $\mu=\Pi_\sharp\nu$ and we have
$$
f(t,x,\nu,a)= F_0(t,x,a) + \pmmu F_1(t,\mu)(x) + F_1(t,\mu),
\quad 
g(x,\mu) = \pmmu G_0(\mu)(x) + G_0(\mu).
$$
For a given flow $t\mapsto \mu_t \in \sP_2(\sD_n)$, 
the terms $F_1(t,\mu_t)$ and $G_0(\mu_T)$ are
independent of the control.  Therefore,
they do not play a role in the control problem of the typical agent,
and without loss of generality, we omit them in the definitions
of the potential MFG. 

The optimization
of  the representative agent in the MFG
is to minimize
\begin{equation}
\label{eq:MFG-cost_(a)}
    J_g(X_0,\alpha,\mu) = \E\left[\int_0^T 
    \left[F_0(t,X^{\alpha}_{\cdot \land t},\alpha_t) + 
    \pmmu F_1(t,\mu_t)(X^{\alpha}_{\cdot\land t}) \right] \,\d t +
    \pmmu G_0(\mu_T)(X^{\alpha}_T)\right]
\end{equation}
over all $\alpha \in \sA$, where
$X^\alpha$ is the controlled state \ref{eq:state} with $\sigma^0=\lambda^0=0$.

The following is a direct consequence of Theorem \ref{thm:pot_game_finite} and Theorem \ref{thm:invariance}.

\begin{Cor}\label{cor:pot_game_finite_(a)}
Suppose Assumptions \ref{asm:state}, \ref{asm:cost} and \ref{asm:invariance} hold,
and assume that $\alpha^*\in\sA$ is a
minimizer of the MFC problem \eqref{eq:MFC_cost} 
without common noise and with running cost given by \eqref{eq:MFC_cost_(a)}. 
Then, with $\mu^*_t=\cL(X^{\alpha^*}_{\cdot\land t})$
the pair $(\alpha^*,\mu^*)$ is a Nash equilibrium for the potential 
MFG with the cost functional \eqref{eq:MFG-cost_(a)}.
\end{Cor}

\subsection{Markovian Cost Without Jumps} 
\label{ssec:markov}

In addition to separability of the cost functionals, we now assume 
that $\lambda=0$.  Then, the state process $X^\alpha$
is continuous in time, and the invariance is proved 
in \cite[Proposition 2.4]{McKean_Vlasov_Formulations}.
Following \cite{briani2018stable},
we also assume that the cost functions in \eqref{eq:MFC_cost_(a)} 
are Markovian, i.e.\ at time $t\in[0,T]$ the dependence 
on the state $X$ is only through the current position $X_t$. 
Hence, the payoff functional \eqref{eq:MFC_cost}
has the form,
\begin{equation}
\label{eq:MFC_cost_(b)}
J(X_0,\alpha) = \E\left[ \int_0^T \left[
\tilde F_0(t,X^\alpha_t,\alpha_t)  +\tilde F_1(t,\cL(X^\alpha_t))
\right]\, \d t \right]
+ \tilde G_0(\cL(X^\alpha_T)),
\end{equation}
for some functions $\tilde{F}_0$, $\tilde{F}_1$ and $\tilde{G}_0$. 

The corresponding MFG cost functional is given by,
$$
J_g(X_0,\alpha,\mu) = \E\left[\int_0^T \left[ \tilde F_0(t,X^\alpha_t,\alpha_t) 
+ \pmmu \tilde F_1(t,\mu_t)(X^\alpha_t)\right]\,\d t +
\pmmu \tilde G_0(\mu_T)(X^{\alpha}_T)\right],
$$
This is exactly the MFG studied in many of 
the earlier studies such as \cite{briani2018stable}.
In particular, it is a standard MFG as the mean field
interaction is only through the law of the current state and not
the control. Then, the measure flow takes values in 
$\sP_2(\R^n)$ instead of $\sP_2(\sD_n)$ or $\sP_2( \sD_n\times U)$. 
Further, by the continuity of $t\mapsto X^\alpha_t$, the flow
$t\mapsto\mu_t\in\sP_2(\R^n)$ is continuous, and the Corollary \ref{cor:pot_game_finite_(a)} holds
with a continuous deterministic flow $\mu^*_t=\cL(X^{\alpha^*}_t)$.

\subsection{Law-Dependent Dynamics}\label{ssec:special-case-law}

We make the same assumptions on the cost functionals $F$ and $G$ of the previous section but allow for the following law-dependence in the drift of the state dynamics, assuming $k=n$,
$$
b(t,x,a,\mu) = a + b_1(t,x,\mu),\quad (t,x,a,\mu)\in[0,T]\times\sD_n\times U \times \sP_2(\sD_n).
$$
Then, under Assumption \ref{asm:law-dep}, the MFC problem of minimizing \eqref{eq:MFC_cost_(b)} subject to dynamics \eqref{eq:state-ext} gives rise to a mean field game with cost functional,
\begin{align*}
    J_g(X_0,\alpha,\nu) = &\E\Bigl[\int_0^T [\tilde F_0(t,X^{\alpha,\mu}_t,\alpha_t) + \pmmu \tilde F_1(t,\mu_t)(X^{\alpha,\mu}_t) \\
    &\qquad - \int_{\R^n\times U} \nabla_a  \tilde F_0(t,\wt x, \wt a) \cdot \delta_\mu b_1(t,\wt x,\mu_t)(X^{\alpha,\mu}_t)\, \nu_t(\d \wt x, \d \wt a) ]\,\d t + \delta_\mu \tilde G_0(\mu_T)(X^{\alpha,\mu}_T)\Bigr],
\end{align*}
which the representative agent minimizes given 
a measure flow $\nu=(\nu_t)_{t\in[0,T]}$. 
We recall that the linear derivative of a vector-valued function is understood entry-wise, $\mu=\Pi_\sharp\nu$, and $X^{\alpha,\mu}$ solves \eqref{eq:MFG-state-law-dep}. This follows by Theorem \ref{thm:pot-game-law-dep} upon noting that in this case we have
$$
\varphi(t,x,y,\mu) = y - b_1(t,x,\mu), \quad (t,x,y,\mu)\in[0,T]\times \sD_n \times B \times \sP_2(\sD_n).
$$

\subsection{Deterministic and Stochastic Formulations}
\label{sec:comparison_Cardaliaguet}

Although the formulation of \cite{briani2018stable}
is seemingly different than that of ours, the associated mean field game that  \cite{briani2018stable} derive is included in our framework.
Indeed, control problems on the space of measures 
have been studied in varying settings. 
In  particular, \cite{McKean_Vlasov_DPP, McKean_Vlasov_Formulations} 
provide a comprehensive overview of different stochastic 
formulations of the problem. In the absence of
common noise and with Markovian dynamics,
these problems also admit  \emph{deterministic}
formulations that directly control the 
evolution of laws via the Fokker-Planck equation
\cite{ bensoussan2015master, briani2018stable,BIRS, cecchin2022weak, lauriere2014dynamic,masoero2019long,SY1,SY2}.

To further explain the connection between 
these stochastic and deterministic formulations 
we specialize to the case considered in \cite{briani2018stable}.
Namely, we assume that
$J$ is given as in \eqref{eq:MFC_cost_(b)}, and consider the simple dynamics
without jumps and common noise,
$$
\d X^\alpha_t =\alpha_t \,\d t + \sigma \d W_t.
$$
In this setting these two formulations
are the following.
\begin{enumerate}[(1)]
    \item \emph{Stochastic (strong) formulation.} 
    On a given probabilistic basis, 
    the objective is to minimize $J(X_0,\alpha)$
    of \eqref{eq:MFC_cost_(b)}
    over controls $\alpha\in\sA$, with above dynamics.
    
    \item \emph{Deterministic formulation.} In this
    formulation the state process
    is the law of the controlled
    diffusions $\mu^\alpha_t=\cL(X^\alpha_t)$,
    and the controls are of feedback type.
    Then, the analogous control problem 
    is to minimize
    \begin{equation*}
        \int_0^T \left( \int_{\R^n} \tilde F_0 (t,x,\alpha(t,x))\,
        \mu^\alpha_t(\d x) \ +\ \tilde F_1(t, \mu^\alpha_t) \right)\,\d t 
        + \tilde G(\mu^\alpha_T),
    \end{equation*} 
    over all measurable feedback functions
    $\alpha :[0,T] \times \R^n \to U$,
    where  the state $\mu^\alpha \in \sP(\R^d)$ solves the 
    Fokker-Planck equation:
    \begin{equation*}
        \partial_t \mu^\alpha + \mathrm{div}(\alpha \mu^\alpha)
        -\frac{\sigma^2}{2}\Delta \mu^\alpha = 0.
    \end{equation*}
    The above is understood in a distributional sense 
    with initial condition $\mu^\alpha_0 = \cL(X_0)$.
\end{enumerate}
The equivalence between 
between these two formulations
is proved using  the \emph{superposition principle} that relates a solution 
of a Fokker–Planck equation to a weak solution of a SDE, see for example 
\cite{lacker2022superposition} 
which establish this fact under much more general dynamics. 
Starting from the deterministic formulation,
this allows us to define a \emph{weak control}  for the stochastic problem
in feedback form with the probability space being part of the control. 
One then appeals to a  mimicking theorem to establish the equivalence of weak control processes and 
weak controls in feedback form. Additional arguments such as boundedness of 
feedback controls are utilized to obtain a \emph{strong control} on a 
fixed probability space. In the present setting, a detailed implementation of
these steps is given in
\cite[Proposition 2.2]{cecchin2022weak}.

\section{Examples }
\label{sec:examples}

\subsection{A Potential MFG of Controls with Price Interaction}

We continue with an example of a mean field game introduced in \cite{bonnans2021schauder}, involving a continuum of traders who interact through both a congestion term and a price variable. In equilibrium, the price is endogenously determined by the distribution of optimal controls. Interpreting the agents as producers rather than traders, similar mean field models have been proposed in the context of optimal power generation and storage within smart grids \cite{alasseur2020extended}. Models of this type are furthermore related to mean field Cournot competition, where producers compete over an exhaustible resource, see \cite{chan2017fracking}.\\

Given a deterministic path for the price $P=(P_t)_{t\in[0,T]}\in\R^k$, a typical trader determines the purchasing rate $\alpha_t$ of $n$ goods. The level of the stock $X^\alpha_t\in\R^n$ is assumed to follow $\d X^\alpha_t=\alpha_t\,\d t +\sigma\,\d W_t$ starting from an initial condition $X^\alpha_0=x\in\R^n$. The typical trader minimizes an expected cost
$$
J_g(X_0,\alpha,\mu)=\E\Bigl[\int_0^T[f_0(t,X^\alpha_t,\alpha_t)+f_1(t,X^\alpha_t,\mu_t)+\phi(t,X^\alpha_t)^T P_t\cdot \alpha_t]\,\d t + g(X^\alpha_T)\Bigr],
$$
where $f_0$ models a purchasing and storage cost, $f_1$ is a congestion term and $\phi:\R^n\to\R^{k\times n}$ is a generic weighing function. For example, taking $n=k$ and $\phi\equiv I_{n\times n}$ allows us to track the prices of each good so that the time $t$ purchasing cost equals $P_t\cdot \alpha_t$. We refer to  \cite{bonnans2021schauder} for further interpretation of the function $\phi$. If $\alpha^*$ denotes an optimal control given $P$, then aggregate demand is
$$
D_t=D_t(P):=\int_{\R^n\times \R^n}\phi(t,x)a\,\cL(X^{\alpha^*}_t,\alpha^*_t)(\d x,\d a),\quad t\in[0,T],
$$
where we made the dependence on the given price path $P$ explicit. If $\Psi(t,z)$ is an inverse demand function, then the supply-demand relation that determines the price level translates into the following equilibrium condition:
$$
P_t=\Psi(t,D_t(P)),\quad \forall t\in[0,T].
$$
Plugging this into the cost functional, we obtain the running cost
$$
f(t,x,a,\nu) = f_0(t,x,a)+f_1(t,x,\mu)+\phi(t,x)\cdot \Psi(t, \int \phi(t,\wt x)\,\wt a \, \nu(\d \wt x,\d\wt a))\,a
$$
for $(t,x,a,\nu)\in[0,T]\times \R^n \times \R^n \times \sP_2(\R^n\times \R^n)$. We therefore see that determining an equilibrium price path turns into solving a mean field game of controls in the sense of Definition \ref{def:MFG_equilibrium}. 

Whenever we can find potentials $F_1:[0,T]\times\R^n\times\sP_2(\R^n)\to\R$, $\Phi:[0,T]\times \R^k\to\R$ such that 
$$
f_1(t,x,\mu) = \delta_\mu F_1(t,\mu),\quad \Psi(t,z) = \nabla_z  \Phi(t,z),\quad (t,x,z,\mu)\in[0,T]\times \R^n\times \R^k \times \sP_2(\R^n)
$$
we can employ Theorem \ref{thm:pot_game_finite} to see that this is a potential mean field game with associated control problem
$$
J(X_0,\alpha) = \E\Bigl[\int_0^T [f_0(t,X^\alpha_t,\alpha_t) + F_1(t,\cL(X^\alpha_t)) + \Phi(t,\E[\phi(t,X^\alpha_t)\alpha_t])]\,\d t + G(X^\alpha_T)\Bigr].
$$
Indeed, we simply note that for any $(t,x,a,\nu)\in[0,T]\times\R^n\times\R^n\times\sP_2(\R^n\times \R^n)$,
$$
\delta_\nu \Psi(t, \nu(\phi(t,\wt x)\,\wt a)) (x,a) = \nabla_z \Phi(t,\nu(\phi(t,\wt x)\,\wt a)) \phi(t, x) a = \Psi(t,\nu(\phi(t,\wt x)\,\wt a) )\phi(t,x)a,
$$
where we write $ \nu(\phi(t,\wt x)\,\wt a):= \int \phi(t,\wt x)\,\wt a \,\nu(\d\wt a,\d\wt a)$.
\begin{Rmk}
\rm{
Possible state constraints common to models of this kind can be accommodated in our approach by \emph{a priori} restricting the class of admissible controls. Infinite horizon cost functions can be handled using the results presented in Appendix \ref{sec:dih}.}
\end{Rmk}

\subsection{Connection to Gradient Flows}

In this section we propose a general connection of Wasserstein Gradient flows to MFC and MFG problems, and we illustrate this approach with three examples. Fix some filtered probability space $(\Omega,\sF,\F,\P)$ 
supporting a Brownian motion $W$ and fix a volatility $\sigma>0$. 
For a sufficiently differentiable function $F:\sP_2(\R^n)\to\R$, consider the  following McKean-Vlasov stochastic differential equation:
\begin{equation}\label{eq:MFLD}
    \d X_t = - \nabla_x \delta_\mu F (\cL(X_t))(X_t)\,\d t + \sigma \d W_t.
\end{equation}
In analogy to finite-dimensional Langevin dynamical systems, \cite{hu2021MFLD} refers to \eqref{eq:MFLD} 
as the \emph{mean field Langevin dynamics}. The corresponding 
non-linear Fokker-Planck equation can be interpreted as a
 gradient flow of measures with respect to the $2$-Wasserstein distance. 
 Recently, these  equations
have received considerable attention due to their connection to noisy gradient descent algorithms in 
deep neural networks \cite{chizat2022MFLD, hu2021MFLD, mei2018MFLD, nitanda2022MFLD}. 

 
 For a linearly convex functional $F$ , \cite{hu2021MFLD}
 proves that the laws of the solution to \eqref{eq:MFLD} converge to 
 the minimizer of the (strictly convex) free energy
\begin{equation*}
    F(\mu) + \frac{\sigma^2}{2} H(\mu),
\end{equation*}
where $H(\mu)$ is the relative entropy of $\mu \in \sP_2(\R^n)$
with respect to Lebesgue measure.  

In the context of mean field control or
games, an alternative connection 
to an optimization problem
is proposed 
for several models, including \cite{carmona2023synchronization, YMMS,yin_synchronization_2010} 
for the Kuramoto synchronization
and \cite{nourian2010synthesis,nourian2011mean} for the Cucker-Smale 
flocking.  In general, we propose the following MFC with a separable structure,
\begin{equation}\label{eq:MFLD_value_function} 
    v(\mu):=\,\inf_\alpha\quad \E \left[\int_0^\infty e^{-\beta t} 
    \left( \frac12 |\alpha_t|^2 + F(\cL(X^\alpha_t))\right)\,\d t\right],\quad \mu\in\sP_2(\R^n),
\end{equation}
where $dX^\alpha_t = \alpha_t\,\d t + \sigma \,\d W_t$, 
$\mu =\cL(X_0)$ is given, and $\beta>0$ is a positive discount factor. 
It is standard that $v$ solves the following dynamic programming equation 
\begin{equation*}
    \beta v(\mu) = -\frac{1}{2}\mu\left(|\nabla_x\pmmu v(\mu)(\cdot)|^2\right) 
    + \frac{1}{2}\mu\left(\Delta_x\pmmu v(\mu)(\cdot)\right) + F(\mu).
\end{equation*}
Whenever the value function is smooth, the optimally controlled state process follows
\begin{equation}\label{eq:MFLD_opt_state}
    \d X_t = - \nabla_x \pmmu v (\cL(X_t))(X_t)\,\d t + \sigma \,\d W_t.
\end{equation}
This is again an equation of type \eqref{eq:MFLD}, except the functional 
$F$ is replaced by the value function $v$. Now, a standard control 
argument shows that linear convexity of $F$ implies linear convexity of the value function $v$, 
suggesting a similar structure of \eqref{eq:MFLD} and \eqref{eq:MFLD_opt_state}. 
Further, note that if $F$ is sufficiently smooth, one can always define a suitable 
control problem such that the optimally controlled state dynamics exactly match \eqref{eq:MFLD}. 
Indeed, this can be achieved by replacing $F$ by
\begin{equation*}
    \wt F(\mu) := \beta F(\mu) 
    + \frac{1}{2}\mu\left(|\nabla_x\pmmu F(\mu)(\cdot)|^2\right) 
    - \frac{1}{2}\mu\left(\Delta_x\pmmu F(\mu)(\cdot)\right).
\end{equation*}
Then, a direct calculation shows that  the value function 
with running cost $\wt F$  is exactly $F$, 
and the optimal state process solves 
the mean field Langevin equation \eqref{eq:MFLD}.

Using our main result, minimizers of the control problem 
give rise to Nash equilibria of an associated MFG
with cost functional,
\begin{equation}
\label{eq:MFLD_pot_game}
     J_{g,\beta}(X_0,\alpha,\mu^*):=
      \E \left[\int_0^\infty e^{-\beta t} \left(\frac12 |\alpha_t|^2 + \pmmu F(\mu^*_t)(X^{\alpha}_t) \right)\,\d t \right],
\end{equation}
and state dynamics $X^\alpha_t = X_0+\int_0^t \alpha_s\,\d s + \sigma W_t$. We continue to illustrating the above discussion with two prominent examples of mean field games.

\subsection{Kuramoto Synchronization} 

The classical  Kuramoto dynamical system 
is a fascinating phenomenological model for
collective synchronization \cite{kuramoto_self-entrainment_1975}. It postulates
the following system of coupled 
ordinary differential equations describing the phases $\theta$ of $N$ oscillators, 
\begin{equation*}
    \frac{\d}{\d t} \theta^k_t = \omega^k + \frac{\kappa}{N} \sum_{j=1}^N \sin(\theta^j_t - \theta^k_t), \quad k=1,..,N.
\end{equation*}
Here $\omega^k$ is the natural frequency of the 
$k$-th oscillator, and $\kappa>0$ denotes the coupling strength,
or the interaction parameter.  The survey article \cite{strogatz_kuramoto_2000} 
provides an excellent  overview and account of the history of this rich subject. 

An intruging property of this model is a phase transition across a critical value of $\kappa$,
from incoherence to spontaneous synchronization. 
The related mean field model  studied in \cite{giacomin2012global}
that shares this feature
is a  McKean-Vlasov equation
obtained by setting $\omega^k$ to zero and adding a Brownian noise
with strength $\sigma >0$.  The resulting dynamics read,
\begin{equation*}
    \d X_t = - \kappa \int_{\T} \sin(X_t-y)\,\cL(X_t)(\d y)\,\d t  + \sigma\,\d W_t,
\end{equation*}
with the one-dimensional torus $\T$ as its state space. 
This flow is  in the form of a mean field Langevin SDE \eqref{eq:MFLD}, for the potential
$$
    F(\mu) := \int_{\T} \int_{\T} \sin^2 ((x-y)/2) \, \mu(\d x) \, \mu(\d y)
    \quad\Rightarrow\quad
   \pmmu F (\mu, x) = \int_{\T} 2\sin^2 ((x-y)/2)  \mu(\d y) .
$$
Then, using the trigonometric identity $2\sin^2(x/2)=1-\cos(x)$, we compute that
$$
    \nabla_x \pmmu F (\mu, x)     = \int_{\T}  \nabla_x  (1-\cos(x-y))  \mu(\d y)
    = \int_{\T} \sin(x-y) \mu(\d y).
$$

The associated MFC problem minimizes
\begin{equation*}
    \inf_\alpha\quad \E 
    \int_0^\infty e^{-\beta t} \left(\frac12 \alpha_t^2 + F(\cL(X^\alpha_t))\right)\,\d t 
\end{equation*}
subject to $X^\alpha_t =  \int_0^t \alpha_s\,\d s + \sigma W_t$. 
In view of Theorem \ref{thm:pot_game_finite}, any minimizer $\alpha^*$ of 
 this MFC gives rise to a mean field Nash 
equilibrium  $\mu^*_t=\cL(X^{\alpha^*}_t)$, in which players minimize,
\begin{equation*}
    \inf_\alpha \quad \E \int_0^\infty 
    e^{-\beta t} \left(\frac12 \alpha_t^2 + \kappa \int_{\T} 2\sin^2\left(\frac{X^{\alpha}_t-y}{2}\right) \, 
    \mu^*_t(\d y) \right)\,\d t .
\end{equation*}
The above MFG approach to synchronization was first proposed 
by Yin, Mehta, Meyn, \& Shanbhag in \cite{yin_synchronization_2010,YMMS}, 
and later used by Carmona \& Graves \cite{carmona_jet_2020} 
to study jet-lag recovery by modeling the alignment with the circadian rhythm. 
These studies and the recent works of Carmona, Cormier, \& Soner 
\cite{carmona2023synchronization}, Cesaroni \& Cirant \cite{cesaroni_stationary_2023}, 
and H{\"o}fer \& Soner \cite{hoefer2024synchronization} 
establish phase transitions in the Kuramoto MFGs analogous to the one 
exhibited by the original and mean field Kuramoto dynamical system, 
providing evidence for a connection between these seemingly very different models.

\subsection{Cucker-Smale Flocking} 
In their classical paper  \cite{cucker2007emergent}, 
Cucker \& Smale propose a model of a flock of $N$ birds in which
the state of the $k$-th bird is described by a vector $(x^k,v^k)\in\R^3\times\R^3$ 
where $x^k$ denotes the location and $v^k$ the velocity. 
With given constants $\kappa>0$, $\rho\geq0$,
they postulate,
\be
\label{eq:CS-ODE}
\frac{\d}{\d t} x^k_t = v^k_t, \qquad
\frac{\d}{\d t} v^k_t = \frac{\kappa}{N}\sum_{k=1}^N \frac{(v^j_t-v^k_t)}{\phi(x_t-\tilde x_t)},
\qquad \phi(x):= (1+|x|^2)^\rho.
\ee

\emph{Flocking} is interpreted as $\max_{k,j=1,..,N}|x^k_t-x^j_t|$ staying bounded 
for all times and 
\cite{cucker2007emergent} proves that flocking occurs if $\rho<1/2$.
Motivated by this model,
Nourian, Caines \& Malham{\'e} \cite{nourian2010synthesis, nourian2011mean}
propose an ergodic mean field game in which the representative player's running cost is given by,
\begin{equation*}
    f((x,v),a,\mu) = \frac12 |a|^2 + \frac{\kappa}{2} \Bigl| \int_{\R^{2n}}
    \frac{v-\tilde v}{\phi(x-\tilde x)}\,\mu(\d \tilde x, \d \tilde v) \Bigr|^2,
\end{equation*}
where $a\in \R^n$, $\mu \in \sP_2(\R^{2n})$. 
Also, \cite{nourian2011mean} observes
that $\rho=0$ reduces to a classical linear-quadratic
control problem and performs a perturbation analysis for small $\rho$.

Although this system is not always a gradient
flow, the velocity equation 
has the form
\be
\label{eq:vk}
\frac{\d}{\d t} v^k_t 
= -\ \frac{\partial}{\partial v^k} \pmmu F(\mu^N_t)(x^k_t,v^k_t),
\ee
where $\mu^N_t$ is the empirical measure of $((x^1_t,v^1_t),. . ,(x^N_t,v^N_t))$  and
$$
 F(\mu) := \frac{\kappa}{4} \int_{\R^{2n}}\int_{\R^{2n}} \frac{|v-\tilde v|^2}{\phi(x-\tilde x)}\,
 \mu(\d x, \d v)\,\mu(\d \tilde x, \d \tilde v), \quad \mu\in\sP_2(\R^{2n}).
$$
Indeed, we calculate that
$$
\pmmu F(\mu)(x,v) = \frac{\kappa}{2} \int_{\R^{2n}} 
\frac{|v-\tilde v|^2}{\phi(x-\tilde x)}\,\mu(\d \tilde x, \d \tilde v)
\   \Rightarrow \ 
 \nabla_v \pmmu F(\mu)(x,v) = \kappa \int_{\R^{2n}} 
\frac{v-\tilde v}{\phi(x-\tilde x)}\,\mu(\d \tilde x, \d \tilde v).
$$
Consistent with this observation,
\cite[Chapter 4.7.3]{carmona_probabilistic_2018-1} suggests
a slightly different MFG than that of 
\cite{nourian2010synthesis}.
It is a potential game corresponding to the MFC 
of \eqref{eq:MFLD_value_function} with 
running cost $F$ and the
$2n$-dimensional state dynamics 
\begin{equation*}
    \d X^\alpha_t = V^\alpha_t\,\d t,\quad \d V^\alpha_t = \alpha_t\,\d t + \sigma\,\d W_t.
\end{equation*}  
As in Subsection \ref{ssect:seperable},
the corresponding MFG minimizes  \eqref{eq:MFLD_pot_game} with
the running cost, 
$$
f((x,v),a,\mu)= \frac12 |a|^2
+ \nabla_v \pmmu F(\mu)(x,v)
=  \frac12 |a|^2
+\kappa \int_{\R^{2n}} 
\frac{v-\tilde v}{\phi(x-\tilde x)}\,\mu(\d \tilde x, \d \tilde v).
$$
In particular, if the value function $v(\mu)$
of this problem is smooth, then the optimal feedback control
is given by $- \nabla_v \pmmu v(\mu)(x,v)$.  Thus, the
optimally controlled state dynamics follow
$$
d X^*_t  = V^*_t\,\d t,\qquad
dV^*_t = - \nabla_v \pmmu v(\mu_t^*)(X^*_t,V^*_t) \, \d t + \sigma \, \d W_t.
$$ 
As in the pure gradient
flow cases, above equation resembles closely the original model \eqref{eq:vk}
with the value function replacing the functional $F$. Additionally,
Theorem \ref{thm:pot_game_finite} implies that
the law of $(X^*,V^*)$ is also a Nash equilibrium of the potential MFG 
\eqref{eq:MFLD_pot_game}.

\subsection{Linear Quadratic Problem}
The special case $\rho=0$ of the Cucker-Smale model
is the well-known linear quadratic problem, which we now discuss.
In this case, the mean field limit  potential is given by,
$$
F(\mu) = \frac12 \int_{\R^n}\int_{\R^n} |x-y|^2\,\mu(\d x) \, \mu(\d y)  
= \frac12 \int_{\R^n}\int_{\R^n}(|x|^2+|y|^2 -2x\cdot y)\, \mu(\d x) \mu(\d y) = \mathrm{Var}(\mu),
$$
for $\mu\in\sP_2(\R^n)$. Then, with $\mathfrak{m}(\mu):= \int y\, \mu(\d y)$,
$$
\pmmu F(\mu)(x) =  \int_{\R^n} |x-y|^2 \mu(\d y) 
\quad\Rightarrow \quad
\nabla_x \pmmu F(\mu)(x) = 2  \int_{\R^n} (x-y) \mu(\d y)
= 2(x- \mathfrak{m}(\mu)).
$$
This implies that the corresponding  mean field Langevin equation is given by,
\begin{equation*}
    \d X_t = - 2 (X_t-\E[X_t])\,\d t + \sigma\,\d W_t.
\end{equation*}
The associated MFC with interaction parameter $\kappa>0$ is
\begin{equation*}
   v(\mu):=  \inf_{\alpha\in\mathcal{A}}\quad \E\int_0^{\infty} 
    e^{-\beta t}\left(\frac{1}{2}|\alpha_t|^2 + \kappa \,
    F(\mathcal{L}(X^\alpha_t))\right)\,\d t,
\end{equation*}
where $X^\alpha_t =  X_0 + \int_0^t \alpha_s\,\d s + \sigma W_t$,
and $\mu =\cL(X_0)$. 
This is a simple
linear, quadratic control problem and can be solved 
explicitly. Indeed, the dynamic programming equation is,
\begin{equation*}
    \beta v(\mu) = -\frac{1}{2}\mu\left(|\nabla_x\pmmu v(\mu)(\cdot)|^2\right) 
    + \frac{\sigma^2}{2}\mu\left(\Delta_x\pmmu v(\mu)(\cdot)\right) 
    + \kappa \mathrm{Var}(\mu).
\end{equation*}
The unique solution is $v(\mu) = a\mathrm{Var}(\mu)+b$ with,
\begin{equation*}
    a = \frac{\sqrt{8\kappa+\beta^2}-\beta}{4},
    \quad b= \frac{\sigma^2 n a }{\beta},
\end{equation*}
so that the optimally controlled state process $X^*$ follows,
\begin{equation*}
   \d X^*_t = -2a(X^*_t-\E[X^*_t]) \,\d t + \sigma\, \d W_t.
\end{equation*}
The stationary measures of this equation are Gaussian 
distributions with arbitrary mean and variance $\sigma^2/(4a)$. 
The associated potential mean field game minimizes
$$
J_{g,\beta}(X_0,\alpha,\mu) 
= \E \int_0^{\infty}e^{-\beta t} \left(\frac{1}{2}|\alpha_t|^2 
+ \kappa \int_{\R} |X^{\alpha}_t-y|^2\,\mu_t(\d y)\right)\,\d t
$$
for a given continuous measure flow $t \in [0,\infty) \mapsto \mu_t \in \sP_2(\R^n)$. 
This mean field game can be solved explicitly, and a direct computation shows 
that any Gaussian distribution with variance $\sigma^2/(4a)$ provides a 
stationary mean field game equilibrium.

\section{Invariance Principle}
\label{sec:invariance}
This section proves the invariance principle
stated in Definition \ref{def:invariance} in a general setting that includes law-dependent dynamics, which leads to an SDE of McKean-Vlasov type.

We fix a probabilistic basis $\gamma$ and an initial condition $X_0 \in \sL^2(\sF^\gamma_0)$. In this section, we allow the controlled state process to be of a general form,
\begin{equation}
\label{eq:invariance-state}
\begin{split}
X^{\alpha}_t = X_0 &+ \int_0^t b(s,X^\alpha_{\cdot\land s},\alpha_s,\mu^\alpha_s)\,\d s + \int_0^t \boldsymbol{\sigma}(s,X^\alpha_{\cdot\land s},\alpha_s,\mu^\alpha_s)\,\d \sW^\gamma_s\\
&+  \int_0^t\int_E
\boldsymbol{\lambda}(s,X^\alpha_{\cdot\land s-},\alpha_s,\mu^\alpha_{s-},\zeta)\, \wt \sN^\gamma(\d s,\d \zeta),
\end{split}
\end{equation}
where $\alpha\in\sA(\gamma)$, $\mu^\alpha_s:=\cL^\gamma(X^\alpha_{\cdot\land s}\,|\,\sG^\gamma_s)$, and $\sG^\gamma$ again refers to the common noise filtration generated by $(W^0,N^0)$. The coefficients are given by functions,
$$
(b,\sigma,\sigma_0):[0,T]\times \sD_n \times U\times \sP_2(\sD_n)
\to \R^n \times \sM_{n\times d}\times
 \sM_{n\times\ell},
$$
$$
(\lambda,\lambda^0): [0,T] \times \sD_n \times U \times \sP_2(\sD_n)\times E
\to \R^n\times \R^n,
$$
and we set $\boldsymbol{\sigma}=(\sigma,\sigma^0)$ and $\boldsymbol{\lambda}=(\lambda,\lambda^0)$.

The cost functional $J^\gamma(\alpha, X_0)$ and value function $v(X_0,\gamma)$ of the MFC are defined by \eqref{eq:MFC_cost} where $X^\alpha$ satisfies \eqref{eq:invariance-state} instead of \ref{eq:state}. Further, let $\fm_2$ be the 2-Wasserstein metric on $\sP_2(\sD_n\times U)$ and 
recall the sup-norm $\| \cdot \|_T$ 
on $\sD_n$ and the notation
${\fs}(\nu)$ for the second moment \eqref{eq:2nd_moment}. 

The following assumption is natural extension of Assumption \ref{asm:state} and ensures strong existence and uniqueness of the general state dynamics \eqref{eq:invariance-state}.

\begin{Asm}\label{asm:invariance-state} 
All coefficients are measurable and  
there exist constants $K,L>0$ and $K(\zeta),L(\zeta)$ 
such that for any $(t,x,x',a,\mu,\mu',\zeta)\in [0,T] \times 
\sD_n^2 \times U\times \sP_2(\sD_n)^2\times E$ we have,
\begin{align*}
|(b,\sigma,\sigma^0)(t,x,a,\mu)-(b,\sigma,\sigma^0)(t,x',a,\mu')| &\leq L(\|x-x'\|_t+\mathfrak{m}_2(\mu,\mu')),\\
|(\lambda,\lambda^0)(t,x,a,\zeta)-(\lambda,\lambda^0)(t,x',a,\zeta)| &\leq L(\zeta)(\|x-x'\|_t+\mathfrak{m}_2(\mu,\mu')),\\
|(b,\sigma,\sigma^0)(t,x,a,\mu)|^2 &\leq K(1+\|x\|_t^2+|a|^2+\mathfrak{s}(\mu)),\\
|(\lambda,\lambda^0)(t,x,a,\mu,\zeta)|^2 &\leq K(\zeta)^2(1+\|x\|_t^2+|a|^2+\mathfrak{s}(\mu)),
\end{align*}
\noindent and the constants $L(\zeta),K(\zeta)$ satisfy $\int_E (L(\zeta)^2+K(\zeta)^2)\, (n^0+n)(\d\zeta)<\infty$.
\end{Asm}

The invariance principle is proved under the following additional assumptions on the coefficients and cost functions. Set $\sX:= \sD_n\times U \times\sP_2(\sD_n\times U)$. 

\begin{Asm}
\label{asm:invariance}
{\rm{
\begin{itemize}
\item (Local Uniform Boundedness).
For any $\ell>0$,  there exists a constant $c_\ell$ such that  for any 
$(t,x,\mu,\zeta)\in[0,T]\times \sD_n\times \sP_2(\sD_n)\times E$ and $|a|, |a'| \le \ell$,
\begin{align*}
 |(b,\sigma,\sigma^0)(t,x,a,\mu) - (b,\sigma,\sigma^0)(t,x,a',\mu)| &\leq c_\ell, \\
  |(\lambda,\lambda^0)(t,x,a,\mu,\zeta) - (\lambda,\lambda^0)(t,x,a',\mu,\zeta)| &\leq c_\ell L(\zeta).
\end{align*}
Further, $\int_EL(\zeta)^{2+\delta}(n^0+n)(d\zeta)<\infty$ for some $\delta>0$.
\item (Quadratic Growth).  $F,G$ are measurable and there is a constant $c_*$ such that
$$
    |F(t,x,a,\nu)| + |G(x,\nu)| \leq C(1+\|x\|_T^2 +|a|^2  +\fs(\nu)),
$$
for every $(t,x,a,\nu)\in [0,T]\times\sX$.
\item (Uniform Continuity).
For any $K > 0$, there is a  
continuous map $m_K : [0,\infty) \to [0,4K]$ with $m_K(0)=0$ 
such that for all  $(t,x,\hat x,\nu,\hat \nu,a, \hat a) \in [0,T]\times \sX^2$,
$$
|F_K(t,x,a,\nu) -F_K(t,\hat x, \hat a, \hat \nu)|+ |G_K(x,\nu) -G_K(\hat x, \hat \nu)| 
\leq m_K(\|x-\hat x\|_T + \fm_2(\nu, \hat \nu)+|a-\hat a|),
$$  
where  $F_K:= (F \wedge K)\vee(-K)$, 
$G_K:= (G \wedge K)\vee(-K)$.
\item (Feasibility). 		
     $v(X_0,\gamma) <\infty$ for every $\gamma$ and $X_0 \in \sL^2(\sF_0^\gamma)$.
\end{itemize}}}
\end{Asm}

\begin{Thm}
\label{thm:invariance}
If the coefficients of an MFC problem satisfy 
Assumptions \ref{asm:invariance-state} and \ref{asm:invariance}, then
it is invariant under probability bases.
\end{Thm}

\begin{Rmk}
 \label{rem:assumptions}   
{\rm{In the diffusion setting,
the invariance is proved in \cite[Theorem 3.6]{Cosso2023}
for `rich enough' probabilistic structures \cite[Section 2.1.3]{Cosso2023}
and in \cite[Proposition 2.4]{McKean_Vlasov_Formulations}.  The 
counter example provided in  \cite[Section 3]{Cosso2023}
shows the necessity of continuity assumptions.
A careful analysis of the proof of our results shows that
local uniform boundedness assumption can be replaced by
$\sL^{p}$ integrability of $X_0$ for some $p>2$, a condition
used in several places, for example \cite{lacker2015}, to overcome similar difficulties.}}
\end{Rmk}

This section is devoted to the proof of Theorem \ref{thm:invariance}.
Let $\sC_m$ be the set of $\R^m$ valued
continuous functions on $[0,T]$, and
recall the measurable spaces $(E,\sE)$ and $(M,\sM)$ related
to the point processes 
given in Appendix \ref{app:point}.

For $\varrho \in \sP_2(\R^n)$ consider
the \emph{canonical}
probability basis 
$$
\gamma^c(\varrho)=(\Omc,\sF^c,\P^c,\F^c,\bW^c,\bN^c),
$$
where
\begin{itemize}
\item $ \Omc = \R^n\times \sC_\ell \times \sC_d \times M^2$;
\item $\P^c$ is probability measure on $\Omc$
given as the product of $\varrho$, the law of a standard 
$(\ell+d)$-dimensional Brownian motion, and  the law
of the point process on $M^2$;
\item The $\sigma$-algebra $\sF^c$  is the $\P^c$ completion of
$\sB(\R^n)\otimes\sB(\sC_{\ell+d})\otimes \sM^{\otimes 2}$;
\item $\bW^c(\om^c)=  (w^0,w)$, $\bN^c(\om^c)=(\pi^0,\pi)$ 
for $\omega^c=(x_0,w^0,w,\pi^0,\pi)\in \Omc$;
\item The  filtration $\F^c = (\sF^c_t)_{t \ge0}$ is
the right-continuous $\P^c$-completion of $\F^\circ$, where
$\sF^\circ_t $ is the smallest $\sigma$-algebra such that
$$
\omega^c=(x_0,w^0,w,\pi^0,\pi)\in \Omc 
\mapsto (x_0,w^0_{\cdot\land t},w_{\cdot\land t},\pi^0(B),\pi(B)) \in \R^{n+\ell+d+2k}
$$
is measurable for every $B=[0,s]\times S$ where $s\leq t$, $S\in\sE$. 
\end{itemize}
We further set $X_0^c(\om^c):=x_0$  for $\omega^c=(x_0,w^0,w,\pi^0,\pi)\in \Omc$ so that 
$\cL^{\P^c}(X_0^c)=\varrho$.

For a basis $\gamma$ and $\eps>0$,
the following definitions are used in the  analysis,
\be
\label{eq:Xi}
\Xi^{\gamma}_t:= (\sW^\gamma_{s \wedge t} , \sN^\gamma_{s\wedge t})_{s \in [0,T]},
\quad
\Xi^{\gamma,\eps}_t:= (W^{0,\gamma}_{s\wedge t}, W^\gamma_{{s \vee \eps}\wedge t} -W^\gamma_\eps, 
\sN^\gamma_{s\wedge t})_{s \in [0,T]},\quad t \in [0,T],
\ee
where $\sN^\gamma_{\cdot\wedge t}$ refers to the restriction of $\sN^\gamma$ to $[0,t]\times E$.

The following is an immediate consequence of the definition. 
\begin{Lem}
\label{lem:canonical}
For any $\varrho \in \sP_2(\R^n)$, $ \gamma \in \Gamma(\rho)$, and 
$X_0 \in \cI^\gamma(\varrho)$, $v(X_0,\gamma) \le v(X_0^c,\gamma^c(\varrho))$.
\end{Lem}
\begin{proof}
For $\alpha^c \in \sA(\gamma^c(\varrho))$, we define $\alpha^\gamma_t(\om):= 
\alpha^c_t( X_0(\om),\sW^\gamma(\om),\sN^\gamma(\om))$
for $t \in [0,T]$, $\om \in \Om^\gamma$.
It is clear that $\alpha \in \sA(\gamma)$ and 
$\cL^{\P^\gamma}(X_0,\sW^\gamma,\sN^\gamma,\alpha^\gamma)= 
\cL^{\P^c}(X_0^c,\sW^c,\sN^c,\alpha^c)$.
Hence, by their definitions,
$J^\gamma(X_0,\alpha^\gamma)= J^{\gamma^c(\varrho)}(X_0^c,\alpha^c)$.
Since for each control in the canonical setting
we have constructed a control in $\gamma$ with same cost,
this implies the claimed inequality.
\end{proof}

In order to prove Theorem  \ref{thm:invariance}, we show
in the next subsections that for any $\varrho \in \sP_2(\R^n)$, $ \gamma \in \Gamma(\rho)$, and 
$X_0 \in \cI^\gamma(\varrho)$ 
the reverse inequality holds,
\begin{equation}
\label{eq:sufficient}
v(X_0,\gamma) 
\ge v(X_0^c,\gamma^c(\varrho)).
\end{equation}
Indeed, this together with Lemma \ref{lem:canonical}
would imply that $v(X_0,\gamma)=v(X_0^c,\gamma^c(\varrho))$.
Since the latter depends only on $\varrho$, we conclude
for all possible bases and initial condition, the value function
depends only on $\varrho$.

The above inequality is established in several steps.
We first reduce it to the case of bounded cost functionals.
The second step is to construct an approximating sequence of piecewise
controls. We then transfer them to strong controls.

\subsection{Reduction to Bounded Cost Functions}
\label{ssec:bounded}

We start with the first reduction and suppress the dependence on $\gamma$.
For positive constants $c,\ell >0$,  set 
$F_{c}:= F\vee (-c)$, $G_c:= G\vee (-c)$
and 
$$
\sA_\ell :=\{ \alpha \in \sA\, :\, |\alpha_t| \le \ell \ \ \text{for all} \ t\in[0,T]\,\}.
$$
Let $J_c$  be the cost functional
given by \eqref{eq:MFC_cost} with $(F,G)$ replaced 
with $(F_c,G_c)$ and define value functions by,
$$
v^\ell(X_0):=\inf_{\alpha \in \sA_\ell} J(X_0,\alpha),\qquad
v^\ell_c(X_0) := \inf_{\alpha \in \sA_\ell} J_c(X_0,\alpha).
$$
We claim that 
$$
\lim_{\ell \to \infty} v^{\ell}(X_0) =v(X_0),\qquad
\lim_{c \to \infty} v^\ell_c(X_0) =v^\ell(X_0).
$$
To prove the first convergence, fix $a_0 \in U$
and for any $\alpha \in \sA$, set
$\alpha^\ell_t := \alpha_t$ if $|\alpha_t|\leq \ell$ and 
$\alpha^\ell_t:=a_0$ otherwise. Then, for large $\ell$, we have $\alpha^\ell \in \sA_\ell$ and
under the assumed regularity conditions, it is standard to show that
$\lim_{\ell \to \infty} J(X_0,\alpha^\ell) = J(X_0,\alpha)$. 
As $v^\ell$ is non-increasing in $\ell$,
this implies that
$\lim_{\ell \to \infty} v^{\ell}(X_0) =v(X_0)$.
Also, for each $\alpha \in \sA_\ell$, one can directly show that 
$\lim_{c\to\infty} J_c(X_0,\alpha)=J(X_0,\alpha)$. 
Since $v^\ell_c$ is non-increasing in $c$, these imply that
$\lim_{c \to \infty} v^{\ell}_c(X_0) =v^{\ell}(X_0)$. 

For a positive constant $m$ set
$F^m_c:= F_c\wedge m$, $G_c^m:= G_c\wedge m$, and
introduce the corresponding value functions by,
$$
v^\ell_{c,m}(X_0) := \inf_{\alpha \in \sA_\ell} J^m_c(X_0,\alpha),
$$
where $J^m_c$ is the cost functional defined 
using $(F^m_c,G^m_c)$.

\begin{Lem}
Suppose that Assumptions \ref{asm:state} and \ref{asm:invariance} holds. Then,
for every $X_0 \in \sL^2(\sF_0)$, and any $c,\ell$,
$$
\lim_{m \to \infty} v^\ell_{c,m}(X_0) =v^\ell_c(X_0).
$$
\end{Lem}

\begin{proof}
Without loss of generality assume $G\equiv0$. We clearly have $v^\ell_{c,m}(X_0)\leq v^\ell_c(X_0)$ for any $m$.
For the reverse inequality let $\alpha^m\in\sA_\ell$ be a $1/m$-optimal control for $v^\ell_{c,m}(X_0)$. Set
$$
\eps(m) := J_c(X_0,\alpha^m)-J^m_c(X_0,\alpha^m).
$$
Then, $\eps(m)\ge 0$ and 
$$
v^\ell_c(X_0) \leq J_c(X_0,\alpha^m) = 
J^m_c(X_0,\alpha^m) + \eps(m) \leq v^\ell_{c,m}(X_0) + 1/m + \eps(m).
$$
Hence, it suffices to show $\lim_{m\to\infty} \eps(m) = 0$. To establish this, we first note that 
the sequence $(X^m)_{m\geq1}$ is bounded in $\sL^2$ due to boundedness of $\alpha^m\in\sA_\ell$. 
In view of the quadratic growth Assumption \ref{asm:invariance} of $F,G$,
\begin{align*}
   0\leq \eps(m) &\leq \E\int_0^T |F_c(t,X^m_{\cdot\land t}, \alpha^m_t,\cL^m_t)
   - F_c^m(t,X^m_{\cdot\land t}, \alpha^m_t,\cL^m_t)| \,\d t\\
   &= \E\int_0^T |F_c(t,X^m_{\cdot\land t}, \alpha^m_t,\cL^m_t)| \chi_{\{|F_c(t,X^m_{\cdot\land t}, \alpha^m_t,\cL^m_t)| > m \}} \,\d t\\
   &\leq c_* \E\int_0^T [1+\|X^m\|_t^2 +|\alpha^m|^2 +\fs(\cL^m_t)] \chi_{\{c_*[1+\|X^m\|_t^2 +|\alpha^m|^2 +\fs(\cL^m_t)]>m\}} \,\d t,
\end{align*}
where we set $X^m := X^{\alpha^m}$ and $\cL^m_t:=\cL^{\alpha^m}_t=\cL((X^m_{\cdot\land t},\alpha^m_t)\,|\,\sG)$. By the boundedness of $\alpha^m$, the $\sL^2$ boundedness of $(X^m)_{m\geq1}$, and basic properties of conditional expectations to deal with the $\cL^m$-terms, it suffices to show uniform integrability of $(\|X^m\|_T^2)_{m\geq1}$,
$$
\lim_{K\to\infty} \ \sup_{m\geq1} \ 
\E[\|X^m\|_T^2 \chi_{\{\|X^m\|_T^2>K\}}] = 0.
$$
Using the assumptions on the SDE coefficients in Assumption \ref{asm:invariance} and \cite[Theorem 2.11]{kunita_jumps}, for any fixed $\beta \in \sA_\ell$ and $\delta>0$ as in Assumption \ref{asm:invariance}, there exists a constant $c$ such that
$$
\E[\|X^m-X^{\beta}\|_T^{2+\delta}] \leq c(1+\E \int_0^T \|X^m-X^{\beta}\|_t^{2+\delta}\,\d t),
$$
and we conclude by Gr\"onwall's inequality that $ \sup_{m\geq1} \E[\|X^m-X^{\beta}\|_T^{2+\delta}] \leq ce^{cT}$.
Set $A(m,K):=\{\|X^m\|_T^2>K\}$. Clearly, by $\sL^2$-boundedness of $(X^m)_{m\geq1}$, 
$$
\sup_{m\geq1} \P(A(m,K)) \leq \frac1K \sup_{m\geq1} \E[\|X^m\|_T^2] \longrightarrow 0, \quad \text{as} \ K\to\infty.
$$
This implies that, for any fixed $\beta\in\sA_\ell$, using H\"older's inequality with $p:=(2+\delta)/2$ and $q:=(2+\delta)/\delta$,
\begin{align*}
    \E[\|X^m\|_T^2 \chi_{A(m,K)}] &\leq 2 \, \E([\|X^m-X^{\beta}\|_T^2 + \|X^{\beta}\|_T^2) \chi_{A(m,K)}] \\
    &\leq 2 \,\E[\|X^m-X^{\beta}\|_T^{2+\delta}]^{1/p} \P(A(m,K))^{1/q} + 2 \, \E[\|X^{\beta}\|_T^2 \chi_{A(m,K)}]\\
    &\leq 2 \, (ce^{cT})^{1/p} \P(A(m,K))^{1/q} + 2 \,\E[\|X^{\beta}\|_T^2 \chi_{A(m,K)}].
\end{align*}
This converges to zero as $K\to\infty$, uniformly in $m$, completing the proof.
\end{proof}

Now suppose that Theorem \ref{thm:invariance} holds 
for all $F$ and $G $ that satisfy not only Assumption \ref{asm:invariance}
but are also bounded. This would imply that for every $\ell,c,m$,
the invariance principle holds for the MFC with functions 
$(F^m_c,G^m_c)$ and with $\sA_\ell$. 
Then, for any $\gamma,\gamma' \in \Gamma(\varrho)$, 
$X_0 \in \cI^\gamma(\varrho)$ 
and $X_0 \in \cI^\gamma(\varrho)$,  $X'_0 \in \cI^{\gamma'}(\varrho)$, 
$$
v^{\ell}_{c,m}(X_0,\gamma)=v^{\ell}_{c,m}(X_0',\gamma').
$$
In view of the above lemma and the discussion preceding it, this implies that
$v(X_0,\gamma)=v(X_0',\gamma')$, proving 
the invariance principle for the original cost functions $(F,G)$.  
Hence, without loss of generality we may assume that $F$, $G$
are bounded and by Assumption \ref{asm:invariance},
they are uniformly continuous with a bounded 
continuous modulus of continuity $\boldsymbol{m}$.

\subsection{Delayed Start}
\label{ssec:delay}

This section introduces a technique presented in \cite{McKean_Vlasov_DPP}. We now fix a stochastic basis $\gamma$, an initial condition $X_0\in\sL^2(\sF^\gamma_0)$, and again suppress the dependence on $\gamma$.
For $\alpha \in \sA$ and $\eps>0$, 
we construct an approximating state process $X^{\alpha,\eps}$
and its cost. We first set
$$
X^{\alpha,\eps}_t = X_0, \qquad \forall \, t \in [0,\eps].
$$
For $t \in [\eps,T]$, we let $X^{\alpha,\eps}$ be the unique strong solution of the
following equation,
\begin{align*}
\nonumber
 X^{\alpha,\eps}_t = X_0 + \int_\eps^ t &
 b(s,X^{\alpha,\eps}_{\cdot \wedge s},\mu^{\alpha,\eps}_s,\alpha_s)\,\d s + \int_\eps^t 
 \boldsymbol{\sigma}(s,X^{\alpha,\eps}_{\cdot \wedge s},\mu^{\alpha,\eps}_s,\alpha_s)\,\d \sW_s\\ 
\tag{SDE$^{\eps}(X_0,\alpha)$}
&+ \int_\eps^t\int_E 
\boldsymbol{\lambda}(s,X^{\alpha,\eps}_{\cdot \wedge s-},\mu^{\alpha,\eps}_{s-},\alpha_s,\zeta)\, 
\wt\sN (\d s,\d \zeta),
\end{align*}
where $\mu^{\alpha,\eps}_s=\cL(X^{\alpha,\eps}_{\cdot \wedge s}\,|\,\sG_s)$. Using $ \cL^{\alpha,\eps}_t:= 
\cL((X^{\alpha,\eps}_{\cdot \wedge t},\alpha_t)\, | \, \sG)$,
we define the approximating cost functional by,
\begin{equation}
\label{eq:app_cost}
J^{\eps}(X_0,\alpha):
=\E\left[\int_0^T F(t,X^{\alpha,\eps}_{\cdot \wedge t}, \alpha_t, \cL^{\alpha,\eps}_t)\ \d t
+G(X^{\alpha,\eps}_{\cdot \wedge T},  \cL^{\alpha,\eps}_T)\right].
\end{equation}

Recall the definition of $\|\alpha\|_{2,T}$ of \eqref{eq:norm}.

\begin{Lem}
\label{lem:app_eps}
Consider sequences $\alpha^k \in \sA$,
$\eps_k>0$.  Then,
$$
\lim_{ k \to \infty}\ 
(\eps_k + \|\alpha^k\|_{2,\eps_k})=0
\quad \Rightarrow \quad \lim_{k \to \infty} \ 
|J(X_0,\alpha^k)-J^{\eps_k}(X_0,\alpha^k)| = 0.
$$
\end{Lem}

\begin{proof}
We first fix $\alpha \in \sA$, $\eps>0$,
and we suppress 
the dependence on $\E|X_0|^2$,
and on the constants in Assumptions \ref{asm:state} and \ref{asm:invariance}.
To simplify the notation, we set
$$
X^k:=X^{\alpha^k},\quad 
Y^k:= X^{\alpha^k,\eps_k},\quad 
\cL^k_t:=\cL^{\alpha^k}_t,\quad
\widehat  \cL^{k}_t:= \cL^{\alpha^k,\epsilon_k}_t.
$$

By Assumption \ref{asm:state}, there is a constant $c_1$
such that,
$$
\E[\|X^{\alpha} -X^{\alpha,\eps}\|^2_\eps]=
\E[\sup_{t\in[0,\eps]}|X^{\alpha}_t -X_0|^2]
\le c_1 (\eps+\|\alpha\|^2_{2,\eps}).
$$
We then use the Burkholder-Davis-Gundy inequality,
\cite[Corollary 2.12]{kunita_jumps}, 
the Lipschitz property of the coefficients 
$(b,\boldsymbol{\sigma},\boldsymbol{\lambda})$, 
and the  Gr\"onwall's inequality to arrive at 
$$
\E[\sup_{t\in[\eps,T]} |X^{\alpha}_t-X^{\alpha,\eps}_t|^2]\
\le\  c_2 \E[|X^{\alpha}_\eps -X^{\alpha,\eps}_\eps|^2] 
=c_2 \E[|X^{\alpha}_\eps -X_0|^2],
$$
for some constant $c_2$.  Hence, there is $c_3$ such that
$$
\E [\|X^{\alpha} -X^{\alpha,\eps}\|^2_T] \le 
\E[\|X^{\alpha} -X^{\alpha,\eps}\|^2_\eps] +
\E[\sup_{t\in[\eps,T]} |X^{\alpha}_t-X^{\alpha,\eps}_t|^2] \le\  c_3 (\eps+\|\alpha\|^2_{2,\eps}).
$$
Since
$\E [\sup_{t\in[0,T]}\fm_2(\cL^{\alpha}_t \, ,\,  \cL^{\alpha,\epsilon}_t )^2] 
\le \E [\|X^{\alpha} -X^{\alpha,\eps}\|^2_T]$ it follows that
$\lim_{k \to \infty}\E[ \zeta_k] =0$, where
$$
\zeta_k:= \|X^k -Y^k\|^2_T
+\sup_{t\in[0,T]} \fm_2(\cL^k_t \, ,\,  \hat \cL^k_t )^2.
$$
Hence, $\zeta_k$ converges to zero in probability.  
Moreover, by the previous subsection,
$$
|F(t, X^{k}_{\cdot\land t}, \alpha^k_t, \cL^{k}_t)
-F(t, Y^{k}_{\cdot\land t}, \alpha^k_t, \widehat  \cL^{k}_t)|+
|G(X^k_{\cdot\land T},\cL^{k}_T) 
- G(Y^k_{\cdot\land T}, \widehat \cL^k_T)|
\le \boldsymbol{m}(\zeta_k),
$$
for every $t \in[0,T]$, where $\boldsymbol{m}$ is the modulus of continuity of $F$ and $G$.
Consequently, the following sequence of random variables
$$
j_k:=\int_0^T [F(t, X^{k}_{\cdot\land t}, \alpha^k_t, \cL^{k}_t)
-F(t, Y^{k}_{\cdot\land t}, \alpha^k_t, \widehat  \cL^{k}_t)]\, \d t 
+ G(X^k_{\cdot\land T},\cL^{k}_T) 
- G(Y^k_{\cdot\land T}, \widehat \cL^k_T),
$$
also converge to zero in probability.
Additionally, 
by the reduction obtained in the previous subsection,
$F$ and $G$ are bounded.  Therefore,
$|j_k|\le
 2 \|F\|_\infty T + 2\|G\|_\infty$.
Then, by  classical arguments we arrive at
$$
\lim_{k \to \infty} |J(X_0,\alpha^k) -
J^{,\epsilon_k}(X_0,\alpha^k)|
=\lim_{k \to \infty} 
\E[| j_k|] =0.
$$
\end{proof}

\subsection{Simple Controls}
\label{ssec:simple_controls}

\begin{Def}
\label{def:simple_conrol}
{\rm{A control process $\alpha \in \sA(\gamma)$ is called}}
simple, {\rm{if 
\begin{equation}
\label{eq:simplealpha}
\alpha_t = a_0 \chi_{[0,t_1]}(t) 
+\sum_{i=1}^{m} \alpha_i \chi_{(t_i,t_{i+1}]}(t), \qquad t \in [0,T],
\end{equation}
for some $t_0=0<t_1=:t_1(\alpha)<.. <t_{m+1} =T$,
$\sF_{t_{i}}^\gamma$-measurable $\alpha_i :\Om^\gamma 
\to U$ and  $a_0\in U$.

Let $\sA_s(\gamma)$ be the set of all simple controls,
and for any $a_0\in U$, $\sA_s(\gamma,a_0)$ is the set of all simple controls that start with $a_0$.}}
\end{Def}

\begin{Lem}
\label{lem:dependence}
Let $\Xi_t^{\gamma,\eps}$ be as in \eqref{eq:Xi}.
For any bases $\gamma, \tilde \gamma$, $\eps>0$, 
$X_0\in \sL^2(\sF^\gamma_0)$, $\alpha \in \sA_s(\gamma)$,
$\tilde X_0 \in \sL^2(\sF^{\tilde \gamma}_0)$, and 
 $\tilde \alpha \in \sA_s(\tilde \gamma)$ with $t_1(\alpha)=t_1(\tilde\alpha)=\eps$,
$$
\cL^{\gamma}(X_0,\alpha,\Xi^{\gamma,\eps}_T)=
 \cL^{\tilde \gamma}(\tilde X_0,\tilde \alpha,\Xi^{\tilde \gamma,\eps}_T) \quad
\Rightarrow \quad
J^{\gamma,\eps}(X_0,\alpha) =
J^{\tilde \gamma,\eps}(\tilde X_0,\tilde \alpha).
$$
\end{Lem}
\begin{proof} Consider the measurable space $\sY:= \R^n\times U^m \times \sC_{\ell} \times \sC_d\times M^2$.
Then,  $\cL^{\gamma}(X_0,\alpha,\Xi^{\gamma,\eps}_T)$
is the measure induced by the map,
$\om \in \Om^\gamma \mapsto 
(X_0(\om),\alpha_1(\om), .. , \alpha_m(\om),\Xi^{\gamma,\eps}_T(\om)) \in \sY$.
Since the state process $X^{\alpha,\eps}$ is a function of the
differentials of $(\sW_t,\sN_t)$ for $t \in [\eps,T]$, 
the independent increment property of these processes implies that the
law of $X^{\alpha,\eps}$ depends only on  
$\cL^{\gamma}(X_0,\alpha,\Xi^{\gamma,\eps}_T)$.  Then,
the statement follows directly from the definition of $J^{\gamma,\eps}$.
\end{proof}

\begin{Lem}\label{lem:approx}
    Suppose Assumption \ref{asm:invariance} holds.
    Then, for any basis $\gamma$, $a_0\in U$, initial condition $X_0\in \sL^2(\sF^\gamma_0)$ and control
    $\alpha \in\sA(\gamma)$, there exists 
    $\alpha^{k} \in \sA_s(\gamma,a_0)$ such that with
    $\eps_k:= t_1(\alpha^{k}) $, 
\be
\label{eq:approx}
\lim_{k\to\infty}\eps_k=0,
\qquad  \text{and}\qquad
\lim_{k\to\infty}  J^{\gamma,\eps_k} (X_0,\alpha^{k})  =J^\gamma(X_0,\alpha).
\ee
\end{Lem}
\begin{proof}
    In the absence of jumps,
this result is proved in \cite[Lemma 2]{McKean_Vlasov_Formulations}, and
we follow its proof closely. We first fix $\alpha \in \sA(\gamma)$, and 
construct a sequence of simple controls $\alpha^k$
approximating $\alpha$ in the $\sL^2(\d\P\otimes \d t)$-norm, \cite[Lemma 4.4]{Liptser_Shiryaev}. Then, the same estimates as in Lemma \ref{lem:app_eps} show that $\lim_{k\to\infty}\E[\|X^{\alpha^k,\eps^k}-X^{\alpha}\|_T^2]=0$. Following the proof of the aforementioned lemma, we use the continuity and boundedness of the functions $F,G$ to conclude that the cost functionals converge. 
\end{proof}

\subsection{Strong Controls}
\label{ssec:strong}

\begin{Def}
\label{def:strong_conrol}
{\rm{A control process $\alpha \in \sA(\gamma)$ is called}}
strong, {\rm{if 
$$
\alpha_t (\om)=  \alpha_t(X_0(\om),
\Xi_t^\gamma(\om)), \qquad t \in [0,T],\ \om \in \Om^\gamma,
$$
where $\Xi_t^\gamma$ is defined in \eqref{eq:Xi}.}}
\end{Def}

We let $\sA_{st}(\gamma)$ be the set of all strong controls.
It is clear that there is a one-to-one connection between
$\sA_{st}(\gamma)$ and $\sA_{st}(\gamma^c(\varrho))$.
Indeed, 
\be
\label{eq:s-st}
\forall \alpha\in\sA_{st}(\gamma), \ \ 
\exists \alpha^c \in \sA_{st}(\gamma^c(\varrho)) \ \
\text{such that}\ \ 
J^{\gamma}(X_0,\alpha)=
J^{\gamma^c(\varrho)}(X_0^c,\alpha^c).
\ee

We next construct a strong control
for each simple control on $\gamma$,
by utilizing a classical stochastic 
transfer lemma, see  for instance \cite[Theorem 6.10]{Kallenberg}. 

\begin{Lem}[Stochastic transfer theorem]
\label{lem:transfer}
Suppose that $(S,\cS)$ is a measurable space and $(T,\cT)$ is a Borel space.
Let $Y$ and $Z$ be random variables on a probability space 
$(\Omega,\sF,\P)$ taking values in $S$ and $T$, respectively. 
Then, for any uniform random variable
$V \in [0,1]$ independent of $Y$,
there exists a measurable function $\psi: [0,1] \times S \to T$ 
such that $\cL^\P(Y,\psi(V,Y)) = \cL^\P(Y,Z)$.
\end{Lem}

The following is the main technical step of the proof, and it is inspired by \cite[Lemma 3]{McKean_Vlasov_Formulations}.
\begin{Prop}
\label{pro:transfer}
For  $X_0\in \sL^2(\sF_0^\gamma)$, $a_0 \in U$, and 
a simple control $\alpha \in \sA_s(\gamma,a_0)$,
there is a strong control
$\alpha^{st} \in \sA_{st}(\gamma) \cap \sA_s(\gamma,a_0)$ such that
with $\eps=t_1(\alpha)$ the followings hold,
$$
t_1(\alpha^{st})=\eps
\quad {\text{and}}\quad J^{\gamma,\eps}(X_0,\alpha) 
=  J^{\gamma,\eps}(X_0,\alpha^{st}).
$$
\end{Prop}
\begin{proof}
We fix $X_0 \in \sL^2(\sF^\gamma_0)$, $a_0\in U$, and
$\alpha \in\sA_s(\gamma,a_0)$. 
With $\alpha$  as in \eqref{eq:simplealpha}, set $\eps:=t_1(\alpha)$ and
$$
\Xi_i:= \Xi^{\gamma,\eps}_{t_i} \in \sC_\ell \times \sC_d \times M^2 =:\sX,
\qquad i=1,..,m,
$$
where $\Xi^{\gamma,\eps}_t$ is as in  \eqref{eq:Xi}.
Then, $(X_0,\Xi_1, .., \Xi_m)$ is independent of $W^\gamma_\eps$.

It is classical that there is a measurable map $\varphi :\R^d \to [0,1]^m$ such that 
components of $\varphi(W^\gamma_\eps)=:\xi =: (\xi_1,. . ,\xi_m)$ are independent uniform variables on $[0,1]$.
Additionally, $(X_0,\Xi_1, .., \Xi_m)$ and $\xi$ are independent.
\vspace{5pt}

\noindent
{\emph{Step 1.}} By the stochastic transfer lemma with $Y=X_0$, $Z=\alpha_1$, and $V=\xi_1$, there exists a measurable $\psi_1:[0,1]\times\R^n\to U$ such that $\cL(X_0,\psi_1(\xi_1,X_0))=\cL(X_0,\alpha_1)$.
Then, for $i=2,..,m$, we apply the transfer lemma with 
$Y=(X_0,\alpha_1, .., \alpha_{i-1}, \Xi_i)$,
$Z= \alpha_i$, and $V=\xi_i$ to construct a measurable map  
$\psi_i: [0,1] \times \R^n\times U^{i-1}\times \sX \to U$
such that 
$$
\cL(X_0,\alpha_1, .., \alpha_{i-1}, \Xi_i, \psi_i(\xi_i,X_0,\alpha_1, .., \alpha_{i-1}, \Xi_i))=
\cL(X_0,\alpha_1, .., \alpha_{i-1}, \Xi_i, \alpha_i).
$$

\noindent
{\emph{Step 2.}} We set $\alpha^{st}_1:= \psi_1(\xi_1,X_0)$.
Then, recursively we define
$$
\alpha^{st}_i:= \psi_i(\xi_i,X_0,\alpha^{st}_1, .. , \alpha^{st}_{i-1}, \Xi_i), \quad i=2, .. , m.
$$
Using this definition and the independence of increments, one may recursively  show that
$$
\cL(X_0,\alpha_1, .., \alpha_{i}, \Xi_i)=
\cL(X_0,\alpha^{st}_1, .., \alpha^{st}_{i}, \Xi_i),\qquad i=1, . ., m.
$$

\noindent
{\emph{Step 3.}} Since $\Theta^{st}:=(X_0,\alpha^{st}_1, .., \alpha^{st}_{m}, \Xi_m)$ and
$\Theta:=(X_0,\alpha_1, .., \alpha_{m}, \Xi_m)$
are $\sF^\gamma_{t_m}$ measurable, they are both independent of the increment
$\Xi^{\gamma,\eps}_T-\Xi_{m}=\Xi^{\gamma,\eps}_T-\Xi^{\gamma,\eps}_{t_m}$.
Also, by the previous step
$\cL(\Theta^{st})=\cL(\Theta)$, hence $\cL(\Theta^{st},\Xi^{\gamma,\eps}-\Xi_m)=\cL(\Theta,\Xi^{\gamma,\eps}-\Xi_m)$. Therefore,
$$
\cL(X_0,\alpha^{st}_1, .., \alpha^{st}_{m},\Xi_T^{\gamma,\eps})
=\cL(X_0,\alpha_1, .., \alpha_{m}, \Xi_T^{\gamma,\eps}).
$$

\noindent
{\emph{Step 4.}} We define a simple, strong control process by,
$$
\alpha^{st}_t :=  a_0 \chi_{[0,t_1]}(t) 
+\sum_{i=1}^{m} \alpha^{st}_i \chi_{(t_i,t_{i+1}]}(t), \qquad t \in [0,T].
$$
In view of Step 2, $\alpha^{st} \in \sA_{st}(\gamma)
\cap \sA_s(\gamma,a_0)$.  
By the previous step, 
$$
\cL(X_0,\alpha^{st}, \Xi_T^{\gamma,\eps})
=\cL(X_0,\alpha, \Xi_T^{\gamma,\eps}).
$$
Then, by Lemma \ref{lem:dependence},
$J^{\gamma,\eps}(X_0,\alpha)= J^{\gamma,\eps}(X_0,\alpha^{st})$
and $\|\alpha\|_{2,T}=\|\alpha^{st}\|_{2,T}$. 
\end{proof}

\subsection{Proof of Theorem \ref{thm:invariance}}
\label{ssec:fnish}

For $\delta>0$, let $\alpha^\delta \in \sA(\gamma)$ be a $\delta$-optimizer of the 
MFC problem on $\gamma$, i.e.,
$$
v(X_0,\gamma) \ge J^\gamma(X_0,\alpha^\delta)-\delta.  
$$
We fix $a_0 \in U$. By Lemma \ref{lem:approx},
there are $\alpha^{k} \in \sA_s(\gamma,a_0)$  such that with $\eps_k=t_1(\alpha^k) \downarrow 0$,
$$
\lim_{k \to \infty}\  J^{\gamma,\epsilon_k}(X_0,\alpha^{k}) 
=J^{\gamma}(X_0,\alpha^{\delta}). 
$$
By Proposition \ref{pro:transfer}, there are 
$\alpha^{k,st} \in \sA_{st}(\gamma) \cap \sA_s(\gamma,a_0)$  satisfying,
$$
J^{\gamma,\epsilon_k}(X_0,\alpha^{k})=
J^{\gamma,\epsilon_k}(X_0,\alpha^{k,st})
\quad\text{and}\quad
t_1(\alpha^{k,st})=\eps_k.
$$
Since $\alpha^{k,st} \in \sA_s(\gamma,a_0)$,
$\alpha^{k,st}_t=a_0$ for every $t \in [0,\eps_k]$
and $\|\alpha^{k,st}\|^2_{2,\eps_k}= |a_0|^2 \eps_k$.
Hence, by
Lemma \ref{lem:app_eps},
$$
\lim_{k \to \infty}\ |J^{\gamma}(X_0,\alpha^{k,st})-
J^{\gamma,\epsilon_k}(X_0,\alpha^{k,st})| =0. 
$$
By \eqref{eq:s-st}, there are $\alpha^{k,st,c} \in \sA(\gamma^c(\varrho))$
so that
$$
J^{\gamma}(X_0,\alpha^{k,st})
= J^{\gamma^c(\varrho)}(X_0^c,\alpha^{k,st,c}).
$$
Using these we choose $\alpha^{\delta,1} \in \sA_s(\gamma)$,
$\alpha^{\delta,2} \in \sA_{st}(\gamma)$,
$\alpha^{\delta,c} \in \sA(\gamma^c(\varrho))$,  and $\eps_*>0$ satisfying,
\begin{align*}
v(X_0,\gamma) &\ge J^\gamma(X_0,\alpha^\delta)-\delta
\ge J^{\gamma,\eps_*} (X_0,\alpha^{\delta,1}) -2\delta
=J^{\gamma,\epsilon_*}(X_0,\alpha^{\delta,2}) -2 \delta\\
&\ge J^{\gamma}(X_0,\alpha^{\delta,2})-3\delta
= J^{\gamma^c(\varrho)}(X_0^c,\alpha^{\delta,c}) - 3\delta
\ge v(X_0^c,\gamma^c(\varrho)) -3\delta.
\end{align*}
As $\delta>0$ is arbitrary, we conclude that
$v(X_0,\gamma)  \ge v(X_0^c,\gamma^c(\varrho))$.
This is exactly \eqref{eq:sufficient} which, as already argued earlier,
implies the invariance principle.

\hfill $\Box$

\appendix

\noindent

\section{Poisson Point Processes} 
\label{app:point}

In this subsection, for the convenience
of the reader we recall standard notations
and terminology from  the classical textbook
of Ikeda \& Watanabe  \cite{Ikeda_Watanabe}
on stochastic processes. See also \cite{barczy2015yamada} for details 
on stochastic differential equations with jumps and the construction of canonical spaces for random measures.

Let $(E,\sE) = (\R^k\setminus\{0\},\sB(\R^k\setminus\{0\}))$, and  
a \emph{counting measure} on $[0,\infty)\times E$
has the form
$$
\pi=\sum_{i\in\cI}\delta_{\{(t_i,p_i)\}},
$$
for some  countable collection (not necessarily distinct nor ordered)
 $(t_i,p_i)_{i\in\cI}\subset [0,\infty)\times E$.
Let $M=M([0,\infty)\times E)$ be the set of all $\{0,1,2,..\}\cup\{\infty\}$-valued counting measures
$\pi$ on $[0,\infty)\times E$ that satisfy $\pi(\{0\}\times E)=0$, $\pi(\{t\}\times E)\leq1$ for all $t>0$ and $\pi([0,T]\times S)<\infty$ for any compact set $S\subset E$. 
Further let $\sM=\sM([0,\infty)\times E)$ denote the smallest $\sigma$-algebra such that 
for each $t\geq0$, $S\in\sE$, the map 
$\pi \in M \ \mapsto \ \pi((0,t]\times S)$ is measurable.

A \emph{($\sigma$-finite) integer-valued random measure} is any 
$(M,\sM)$-valued random variable. A \emph{Poisson random measure} 
on $[0,\infty)\times E$ is a random measure that satisfies
\begin{enumerate}[(i)]
    \item For $B\in\sB[0,\infty)\otimes\sE$, $\pi(B)$ is Poisson 
    with the \emph{intensity measure}  $\lambda(B)=\E[\pi(B)]$;
    \item For disjoint sets $B_1,..,B_n\in\sB[0,\infty)\otimes\sE$, 
    the random variables $\pi(B_1),..,\pi(B_n)$ are mutually independent.
\end{enumerate}

A function $p:D(p)\subset(0,\infty)\to E$ is called a \emph{point function on $E$}, 
if its \emph{domain} $D(p)$ is countable and 
$\{s\in D(p)\cap(0,t] : p(s)\in S\}$ finite for every $t>0$ and 
a compact $S\subset E$. For every point function on $E$
we associate a counting measure $N_p$ on $[0,\infty)\times E$ by, 
\begin{equation*}
    N_p(\d t,\d \zeta) = \sum_{s\in D(p)} \delta_{\{(s,p(s))\}}(\d t,\d \zeta).
\end{equation*}
Note that $N_p(\{0\}\times E)=0$. 
We equip the space of all point functions with the $\sigma$-algebra 
which is generated by the family of maps $p\mapsto N_p((0,t]\times S)\in[0,\infty]$, 
for $t>0$, $S\in\sE$. 

A \emph{point process} is a randomized point function $p(\omega,t)$, 
measurable in the sense introduced before, and it
is called a \emph{Poisson point process} if the associated random counting measure $N_p$ is 
a Poisson random measure. Finally, it is called \emph{stationary}, 
if the intensity measure of $N_p$ is given by the product of Lebesgue measure with some 
$\sigma$-finite measure $n(\d\zeta)$ on $(E,\sE)$:
\begin{equation*}
    \hat N_p(\d t, \d\zeta) = \d t \otimes n(\d \zeta).
\end{equation*}
The deterministic measure $\hat N_p$ is  the \emph{compensator} or \emph{dual predictable projection} 
of the Poisson random measure $N_p$. The random measure 
$\wt N_p := N_p - \hat N_p$ is called the \emph{compensated Poisson random measure}. 
With these definitions, $\sigma$-finite integer-valued random measures are in one-to-one correspondence with point processes.

On a filtered probability space $(\Omega,\sF,\F,\P)$,
a Poisson point process $p=p(\omega,t)$ on $\Omega$ is called an
\emph{$\F$-Poisson point process} if 
for every $t, h>0$, and $S\in\sE$,
$N_p((0,t]\times S)$ is $\sF_t$-measurable,
and the increments $(N_p((0,t+h]\times S)-N_p((0,t]\times S))$ 
are independent of $\sF_t$.
For such processes, the stochastic integral 
\begin{equation*}
    \int_0^t\int_E f(\omega,s,\zeta)\, \wt N_{p(\omega)}(\d s,\d \zeta)
\end{equation*}
is defined for any real function $f$ on $\Omega\times[0,\infty)\times E$ 
that is $\sP\otimes\sE$-measurable, where $\sP$ is the predictable $\sigma$-field on $\Omega\times[0,\infty)$, and which satisfies
\begin{equation*}
    \P\left(\int_0^t\int_E |f(\omega,s,\zeta)|^2\, n(\d\zeta)\,\d s<\infty \right)=1,\quad t>0.
\end{equation*}
Moreover, the identity
\begin{equation*}
    \int_0^t\int_E f(s,\zeta)\, \wt N_p(\d s,\d \zeta) = \int_0^t\int_E f(s,\zeta)\, N_p(\d s,\d \zeta) - 
 \int_0^t\int_E f(s,\zeta)\, n(\d\zeta)\,\d s,
\end{equation*}
with the right side being standard Lebesgue integrals, only holds when the terms on the right side separately converge. 
In general, the left side is defined as an (localized) $\sL^2$-limit. 
For details we refer to Chapter II.3 \cite{Ikeda_Watanabe} and \cite{kunita_jumps} 
which present a more general integration theory with respect to random compensators.

\section{Discounted Infinite Horizon} 
\label{sec:dih}

We consider the discounted infinite horizon problems 
with a  fixed discount factor $\beta>0$.

\begin{Def}[$\beta$-admissible controls] \label{def:adm_controls_beta}
    {\rm{For a probability basis $\gamma$, 
    the set of}} $\beta$-admissible controls $\sA_\beta(\gamma)$ 
    {\rm{is the set of all $\F^\gamma$-progressively measurable processes 
    $\alpha:\Omega^\gamma\times [0,\infty)\to U$ satisfying
    \begin{equation}\label{eq:beta-controls}
        \E^{\gamma} \int_0^\infty e^{-\beta t} [\|X^\alpha\|^2_t + |\alpha_t|^2]\,\d t< \infty,
    \end{equation}
    for any initial condition $X_0\in \sL^2(\sF^\gamma_0)$. 
    Here $X^\alpha$ is the unique strong solution to \ref{eq:state}}}.
\end{Def}

The discounted infinite horizon MFC problem is to minimize
\begin{equation}\label{eq:MFC_beta}
    J_{\beta}^\gamma(X_0,\alpha):=\E^{\gamma}
    \int_0^\infty e^{-\beta t} 
    F(t,X^{\alpha}_{\cdot\land t},\alpha_t,\cL^\alpha_t)\,\d t,
\end{equation}
over all $\alpha \in \sA_\beta(\gamma)$, where as before 
$\cL_t^\alpha:=\cL^{\gamma}((X^\alpha_{\cdot\land t},\alpha_t)\,|\, \sG^\gamma)$. 

For a given $\varrho \in \sP_2(\R^n)$, 
$\cI_\beta^\gamma(\varrho)$ is the
set of all $\sF_0^\gamma$-measurable  $X_0$ such that
$\cL^{\gamma}(X_0)=\varrho$, and $\Gamma_\beta(\varrho)$ is the 
set of all $\gamma$ such that both $\sA_\beta(\gamma)$ and $\cI_\beta(\varrho)$ are non-empty.

\begin{Rmk}
{\rm{In many applications, the controlled diffusion follows 
$\d X^\alpha_t = \alpha_t\,\d t + \sigma\,\d W_t$. 
In this case, $\alpha \equiv 0$ always belong to 
$\sA_\beta(\gamma)$ and the additional 
requirement about the existence of a control
is always satisfied.}}
\end{Rmk}

The following result is obtained
by following the proof of Theorem \ref{thm:invariance}
\emph{mutatis mutandis}.

\begin{Thm}\label{thm:invariance_beta}
Suppose that  Assumptions \ref{asm:state}, \ref{asm:invariance} hold.  
Then, for any $\varrho\in\sP_2(\R^n)$,
$$
v_\beta(X_0,\gamma):=\inf_{\alpha\in\sA_\beta(\gamma)} \,\, 
J_\beta^{\gamma}(X_0,\alpha)=
\inf_{\gamma \in \Gamma_\beta(\varrho)} 
\inf_{X_0 \in \cI_\beta^\gamma(\varrho)} \inf_{\alpha \in \sA_\beta(\gamma)}
J_{\beta}^\gamma(X_0,\alpha)=:V_\beta(\varrho),
$$
for all $\gamma\in\Gamma_\beta(\varrho)$ and $X_0\in\cI_\beta^\gamma(\varrho)$.
\end{Thm}

We continue by defining the MFG in this context.
Let $f$ be as in \eqref{eq:pot-game-cost}, and set
$$
J_{g,\beta}^\gamma(X_0,\alpha,\nu^*):=
\E^{\gamma} \int_0^\infty e^{-\beta t} 
f(t,X^{\alpha}_{\cdot\land t},\alpha_t,\nu^*_t)\,\d t,
$$

\begin{Def}\label{def:MFG_beta_equilibrium}
    {\rm{For a basis  $\gamma$ and an initial condition $X_0\in \sL^2(\sF^\gamma_0)$,
    a pair $(\alpha^*,\nu^*)$ is an}} 
    discounted MFG of controls Nash equilibrium on $\gamma$ {\rm{if
    the followings hold:
    \begin{enumerate}[(i)]
        \item $\alpha^* \in \mathcal{A}_\beta(\gamma)$ satisfies
        $ J_{g,\beta}^\gamma(X_0,\alpha^*,\nu^*) = \inf_{\alpha \in \sA_\beta(\gamma)} 
        J_{g,\beta}^\gamma(X_0,\alpha,\nu^*)$,
        \item  $\nu^*_t = \cL^{\gamma}((X^{\alpha^*}_{\cdot\land t}, \alpha^*_t) \,|\, \sG^\gamma\, )$ 
        $\Prob^\gamma$-a.s.\ for all $t\geq0$.
    \end{enumerate}}}
\end{Def}

A straightforward modification of the proof of Theorem 
\ref{thm:pot_game_finite},  using the integrability condition 
\eqref{eq:beta-controls} and Theorem \ref{thm:invariance_beta}, shows the following.

\begin{Thm}\label{thm:pot_game_beta}
    Suppose that Assumptions \ref{asm:state}, \ref{asm:cost} hold, and the MFC problem is invariant under probability bases as in Theorem \ref{thm:invariance_beta}. Let $\gamma$ be a probabilistic basis with non-empty
    $\sA_\beta(\gamma)$ and $X_0\in \sL^2(\sF^\gamma_0)$.
    If $\alpha^*\in\sA_{\beta}(\gamma)$ is a
    minimizer of the MFC problem \eqref{eq:MFC_beta}, then 
    $(\alpha^*, \nu^*)$ is a discounted MFG of controls 
    Nash equilibrium for the potential game on the basis $\gamma$,
    where  $\nu^*_t := \cL^\gamma((X^{\alpha^*}_{\cdot\land t},
    \alpha^*_t)\,|\, \sG^\gamma\, )$.
\end{Thm}

\bibliographystyle{abbrvnat}
{\small
\bibliography{potential_games.bib}}
\end{document}